\documentclass{article}
 \usepackage[backref,colorlinks,linkcolor=red,anchorcolor=green,citecolor=blue]{hyperref}
\usepackage{amsfonts,amssymb}
\usepackage{amsmath,amsthm}
\usepackage{enumerate}
\RequirePackage[left=1in,right=1in,top=1in,bottom=1in]{geometry}
\usepackage[utf8]{inputenc}
\usepackage{color}

\numberwithin{equation}{section}

\theoremstyle{plain}
\newtheorem{thm}{Theorem}[section]
\newtheorem{lem}{Lemma}[section]
\newtheorem{prop}{Proposition}[section]

\newtheorem{defn}{Definition}[section]
\theoremstyle{definition}
\newtheorem{rem}{Remark}[section]

\newcommand{\gradt}{\|\nabla\theta\|}

\newcommand{\rd}{\mathbb{R}^d}

\begin{document}
\title{Strong solutions to a modified Michelson-Sivashinsky equation}
\date{\today}
\author{Hussain Ibdah\footnote{Department of Mathematics, Texas A\&M Univeristy, College Station, TX 77843, USA (hibdah@math.tamu.edu)}}
\maketitle
\begin{abstract}
We prove a global well-posedness and regularity result of strong solutions to a slightly modified Michelson-Sivashinsky equation in any spatial dimension and in the absence of physical boundaries. Local-in-time well-posedness (and regularity) in the space $W^{1,\infty}(\rd)$ is established and is shown to be global if in addition the initial data is either periodic or vanishes at infinity. The proof of the latter result utilizes ideas previously introduced by Kiselev, Nazarov, Volberg and Shterenberg to handle the critically dissipative surface quasi-geostrophic equation and the critically dissipative fractional Burgers equation. Namely, the global regularity result is achieved by  constructing a time-dependent modulus of continuity that must be obeyed by the solution of the initial-value problem for all time, preventing blowup of the gradient of the solution. This work provides an example where regularity is shown to persist even when a-priori bounds are not available.
\end{abstract}
\section{Introduction}
\subsection{Background.}
We study the following nonlinear, nonlocal parabolic initial-value problem 
\begin{equation}\label{eq1}
\begin{aligned}
&\partial_t\theta-\nu\Delta\theta=\lambda|\nabla\theta|^{p}+\mu(-\Delta)^{\alpha}\theta,\\
&\theta(0,x)=\theta_0(x),
\end{aligned}
\end{equation}
where $\nu>0$, $\alpha\in(0,1/2)$, $p\in[1,\infty)$, $\mu>0$, $\lambda\in\mathbb{R}$ and $\theta:[0,\infty)\times\rd\rightarrow\mathbb{R}$ is a scalar. We show that if $\theta_0\in W^{1,\infty}(\rd)$ and is periodic (with arbitrary period $L>0$ in every direction) or vanishes at infinity, then there is a unique globally regular solution to \eqref{eq1} satisfying the bound
\[
\|\nabla\theta(t,\cdot)\|_{L^{\infty}}\leq Be^{C_0t},
\] 
where $B$ depends only on $\|\theta_0\|_{W^{1,\infty}}$ and $C_0$ depends on $B,\nu,\alpha,d,\mu$ (see Theorems \ref{shortimeexis} and \ref{mainthm}, below).

Let us start by discussing the motivation behind this work and provide some background information. One of the outstanding questions in the analysis of partial differential equations is whether the Kuramoto-Sivashinsky (KS) equation develops a singularity in finite time or whether solutions arising from smooth enough initial data remain smooth for all time (in spatial dimension $d\geq2$). In its scalar form, this equation reads
\begin{equation}\label{KSE}
\partial_t\theta(t,x)+\Delta^2\theta(t,x)+\frac{1}{2}|\nabla\theta(t,x)|^2+\Delta\theta(t,x)=0.
\end{equation}
In spatial dimension $d=1$, the solution to the initial value problem associated with \eqref{KSE} (in the periodic or whole space setting) does not develop any singularities in finite time starting from smooth enough initial data $\theta_0$, see for instance \cite{NicoSheurer1984, Tadmor1986}. In dimensions $d=2,3$, and under the assumption of radially symmetric initial data in an annular region with homogenous Neumann boundary conditions, global regularity was proven in \cite{BBT2003}. Nevertheless, the question of global well-posedness of the IVP associated with \eqref{KSE} remains open, in the large, for arbitrary smooth enough initial data when the spatial dimension $d$ is larger than one. 

The KS equation was derived independently by Sivashinsky \cite{Sivashinsky1977} as a model for flame propagation (see also \cite{MichelsonSivashinsky1977}), and by Kuramoto \cite{Kuramoto1978} in the context of a diffusion-induced chaos in a chemical reaction system (see also \cite{KuramotoTsuzuki1975, KuramotoTsuzuki1976}). The original model derived by Sivashinsky in \cite{Sivashinsky1977} and discussed in \cite{MichelsonSivashinsky1977} reads 
\begin{equation}\label{fullcombs}
\partial_t\theta+4(1+\epsilon)^2\Delta^{2}\theta+\epsilon\Delta\theta+\frac{1}{2}|\nabla\theta|^2=(1-\sigma)(-\Delta)^{1/2}\theta,
\end{equation}
where $\sigma\in(0,1)$ is the coefficient of thermal expansion of the gas, $\epsilon=(L_0-L)/(1-L_0)$, with $L$ being the Lewis number of the component of the combustible mixture limiting the reaction, and  $L_0<1$ being the critical Lewis number depending on various physical properties of the mixture. Here, $(-\Delta)^{\alpha}$, $\alpha\in(0,1)$, is the nonlocal operator, whose Fourier symbol is given by $|k|^{2\alpha}$. Equivalently, it can be represented in terms of the singular integral
\begin{equation}\label{singintintro}
(-\Delta)^{\alpha}\theta(x)=C_{d,\alpha}P.V.\int_{\mathbb{R}^d}\frac{\theta(x)-\theta(x-z)}{|z|^{d+2\alpha}}\ dz,
\end{equation}
for $\alpha\in(0,1)$, sufficiently regular $\theta$, and $C_{d,\alpha}>0$ being a normalizing constant, degenerating as $\alpha\rightarrow 0^+$ or $1^-$. When $\epsilon>0$, upon rescaling, one formally recovers equation \eqref{KSE} from \eqref{fullcombs} by setting $\sigma=1$. Much of the analysis done in the literature is carried out for the case when $\epsilon>0$ and $\sigma=1$. To the best of our knowledge, no rigorous mathematical treatment for the case $\sigma\neq1$ has been done. Furthermore, when $L>L_0$ ($\epsilon<0$), asymptotic analysis leads to dropping out the hyperviscous term $\Delta^2$ in \eqref{fullcombs}, and the instabilities in the flame in this case arise as a consequence of thermal expansion on its own \cite{MichelsonSivashinsky1977, Sivashinsky1977}, and one gets (upon rescaling)
\begin{equation}\label{MS}
\partial_t\theta-\Delta\theta-(-\Delta)^{1/2}\theta+\frac{1}{2}|\nabla\theta|^2=0.
\end{equation}
In other words, it is physically possible to have $\epsilon<0$; we refer the reader to the survey articles \cite{Matalon2007, Sivashinsky1983} for further insight regarding the physical role of the parameters in \eqref{fullcombs} in the theory of combustion.

Equation \eqref{MS} is called the Michelson-Sivashinsky (MS) equation. It is a refined combustion model based on the Darrieus–Landau flame stability analysis, and was also recently derived in \cite{Navin2010, TP2019}. Several computational studies were performed on the periodic one-dimensional version of \eqref{MS}, see for instance \cite{GS1990, MS1982, MichelsonSivashinsky1977, Pumir1985}, where typical turbulence-induced chaotic behavior was noted. Numerical observations have led several authors to consider special solutions of \eqref{MS} in the one-dimensional case (see, for instance, \cite{KO2013,OGKP1997,Renardy1987,TFH1988489} and the references therein). However, the global regularity of the one-dimensional version of \eqref{MS} does not present any mathematical challenges. Indeed, one has a-priori control over the $H^1$ norm of the solution, which can be bootstrapped to control higher order Sobolev norms, with the nonlocal part causing at most growth in time, but not blow up.

In dimensions higher than one, one runs into the same technical difficulties as in the KS equation. Namely, no a-priori bound, not even in $L^2$, can be obtained, due to the nonlinear term. Thus, one can only prove short-time existence, uniqueness and regularity via standard arguments for smooth enough initial data. On the other hand, the fact that the dissipative operator in the KS, $\Delta^2$, is replaced by the standard Laplacian, $(-\Delta)$, in the MS equation \eqref{MS}, there might be hope in controlling the Lipschitz constant of the solution to \eqref{MS} (i.e., proving a ``maximum principle'' for the gradient of the solution to \eqref{MS}), which can then be bootstrapped to control higher order derivatives, analogous to the viscous Burgers equation. This was also the basis of the recent work \cite{LariosYamazaki2019}, where the authors propose a modification of the KSE in its vectorial form. Namely, by replacing hyperviscosity with standard Laplacian in one component, they were able to bootstrap the resulting maximum principle and show that smoothness persists under evolution.

A rather ingenious method developed fairly recently by Kiselev, Nazarov and Volberg \cite{KNV2007} (see also \cite{KNS2008}) was used to obtain a maximum principle for the critically dissipative surface quasi-geostrophic (SQG) equation (and the fractal Burgers equation). Evolution under the critically dissipative SQG equation (when $d=2$) is described by
\begin{equation}\label{SQG}
\begin{cases}
	&\partial_t\theta+(-\Delta)^{1/2}\theta+(u\cdot\nabla)\theta=0,\\
	&u=(u_1,u_2)=(-R_2\theta,R_1\theta),
\end{cases}
\end{equation}
where $R_1$, $R_2$ are the usual Riesz transforms in $\mathbb{R}^2$. Even though \eqref{SQG} has a maximum principle of the form $\|\theta(t,\cdot)\|_{L^{\infty}}\leq\|\theta_0\|_{L^{\infty}}$, this control although useful, does not necessarily prevent blowup in general, and one would require control of a stronger norm in order to address the global existence of smooth solutions in the positive direction. The elegant work in \cite{KNS2008,KNV2007} introduced techniques that allow one to compare dissipation, (gradient) nonlinearity and nonlocality in the local (pointwise) setting, without any a-priori assumptions other than short-time existence and regularity. The main idea is to show that if the initial data has a certain modulus of continuity (see Definition \ref{defmod}, below), and if the solution is guaranteed to be smooth for short time, then preservation of the modulus of continuity on some non-degenerate time interval $[0,T]$ implies control of the Lipschitz constant of the solution on that interval, which in many cases is sufficient to prevent blowup of higher order norms. The difficulty lies in constructing a modulus of continuity that is able to (locally) balance dissipation with the instabilities that may arise from nonlinearity and nonlocality for all time. In many cases this is not a trivial task, see for instance \cite{DKSV2014, Kiselev2011, KN2010, MX2019} and the references therein where this program was expanded and built upon in several other scenarios.

Such techniques rely upon pointwise estimates, and so it is crucial to be able to
\begin{enumerate}[(1)]
\item make sense of the PDE in the classical way,
\item make sure the solution enjoys parabolic regularity $C_t^1C_x^2$,
\item obtain pointwise estimates of all terms in the PDE, preferably via quantifying continuity of such terms in terms of H\"older estimates, or the modulus of continuity itself,
\item have a regularity criterion in terms of the Lipschitz constant of the solution.
\end{enumerate}

That being said, in order to study the evolution of moduli of continuity under \eqref{MS} (or even formally obtain a maximum principle), a pointwise upper bound for the nonlocal part must be obtained, ideally in terms of the modulus of continuity being studied. In fact, as will be demonstrated later on, all what one really needs is a bound that does not exceed a constant multiple of $\|\nabla\theta\|_{L^{\infty}}$. However, this does not seem to be possible: the square root of the Laplacian has the representation
\[
(-\Delta)^{1/2}\theta=\sum_{i=1}^dR_i\partial_i\theta,
\]
with $\{R_i\}_{i=1}^d$ being the standard Riesz transforms, and it is well known that $L^{\infty}$ is a bad space for those operators, see for instance \cite{Stein1970book}. That is, even when $\theta$ has a modulus of continuity and its Lipschitz constant is under control, no information can be obtained about $(-\Delta)^{1/2}\theta$ in terms of the controlled quantities; we refer the reader to \cite{Stein1961, Wheeden1968} for a classical characterization of the singular integral \eqref{singintintro}, and \cite{CaffarelliSilvestre2007} for a more recent one. Nevertheless, see the conclusion of this paper for further remarks about a possible remedy to this situation. This has led the author to consider a slightly weaker model than \eqref{MS}, namely equation \eqref{maineq}, below.
\subsection{Main Results.}\label{secmainresult}
 With the previous remarks in mind, replacing the nonlocal part of equation \eqref{MS} with $(-\Delta)^{\alpha}$, where $\alpha\in(0,1/2)$, allows one not only to locally bound the corresponding nonlocality (Lemma \ref{bdfraclap}, below), but also to obtain a continuity estimate. Indeed, if $\theta\in C^{0,\beta}$ with $0<2\alpha<\beta\leq1$, then $(-\Delta)^{\alpha}\theta\in C^{0,\beta-2\alpha}$ \cite{Silvestre2007}. Similarly, we show in Lemma \ref{fraclapmod}, below, that while the operator $(-\Delta)^{\alpha}$ doesn't quite preserve an abstract modulus of continuity, it doesn't distort it too much either. This allows us to control the nonlocality and prove that dissipation will prevail, thereby proving that strong (and hence classical) solutions exist and are unique for all time. The power of the nonlinearity does not seem to introduce any extra complications in the proof in the absence of physical boundaries, see for instance \cite{PSbook2019} for various blowup results for viscous Hamilton-Jacobi equations in the presence of boundaries. Thus, in this work we study the initial value problem associated with 
\begin{equation}\label{maineq}
\partial_t\theta(t,x)-\nu\Delta\theta(t,x)=\lambda\left|\nabla\theta(t,x)\right|^p+\mu(-\Delta)^{\alpha}\theta(t,x),
\end{equation}
where $\nu>0$, $\alpha\in(0,1/2)$, $p\in[1,\infty)$, $\mu>0$ and $\lambda\in\mathbb{R}$, with no further restrictions on such parameters. We study evolution under equation \eqref{maineq} starting from a $\theta_0\in W^{1,\infty}(\mathbb{R}^d)$ and we look for strong solutions on an interval of time $[T_1,T_2]$. By a strong solution, we mean
 \begin{defn}\label{strongsol}
Let $T_2>T_1$, and suppose $\theta_0\in W^{1,\infty}\left(\mathbb{R}^d\right)$. We say $\theta$ is a strong solution to \eqref{maineq} on $[T_1,T_2]$ corresponding to $\theta_0$ if $\theta\in C\left([T_1,T_2];W^{1,\infty}\left(\mathbb{R}^d\right)\right)$ and 
\begin{align*}
\theta(t,x)=&\int_{\mathbb{R}^d}\Psi(t-T_1,x-y)\theta_0(y)\ dy+\lambda\int_{T_1}^t\int_{\mathbb{R}^d}\Psi(t-s,x-y)\left|\nabla\theta(s,y)\right|^pdyds\\
&+\mu\int_{T_1}^t\int_{\mathbb{R}^d}\Psi(t-s,x-y)\left(-\Delta\right)^{\alpha}\theta(s,y)\ dy\ ds, \quad (t,x)\in[T_1,T_2]\times\mathbb{R}^d,
\end{align*}
where $\Psi$ is the $d$-dimensional heat kernel,
\[
\Psi(s,y):=(4\pi\nu s)^{-d/2}\exp\left(\frac{-|y|^2}{4\nu s}\right),\quad (s,y)\in\mathbb{R}^+\times\mathbb{R}^d.
\]
\end{defn}
\begin{rem}
Lemma \ref{bdfraclap}, below, allows us to make sense of $(-\Delta)^{\alpha}\theta$ as an $L^{\infty}$ function.
\end{rem}

Using standard (classical) properties of the heat kernel, one can show that strong solutions satisfy the initial condition in the sense
\begin{equation}\label{IC}
\lim_{t\rightarrow T_1^+}\|\theta(t,\cdot)-\theta_0\|_{W^{1,\infty}}=0,
\end{equation}
and are classical solutions to the PDE \eqref{maineq}. By classical, we mean that they are once continuously differentiable in time and twice in space on the set $(T_1,T_2]\times\rd$ and satisfy \eqref{maineq} in the pointwise sense. In addition, their time derivatives have the regularity $\partial_t\theta\in L^1([T_1,T_2];L^{\infty}(\rd))$, and a regularity criterion in terms of the Lipschitz constant of the solution should not be surprising. That is, we first establish the following local well-posedness result.
\begin{thm}\label{shortimeexis}
Let $d\in \mathbb{N}$, $\nu>0$, $\alpha\in(0,1/2)$, $\lambda\in\mathbb{R}$, $\mu>0$, $p\in[1,\infty)$ and $\theta_0\in W^{1,\infty}\left(\mathbb{R}^d\right)$ with no further restrictions. Then there is a $T_0=T_0(\theta_0,d,\alpha,p,\nu,\mu,\lambda)>0$ and a strong solution $\theta$ to \eqref{maineq} on $[0,T_0]$  corresponding to $\theta_0$ and depending continuously on the initial data in the $W^{1,\infty}\left(\mathbb{R}^d\right)$ norm. In particular, it is the only strong solution. Furthermore, if $\theta$ is the strong solution corresponding to $\theta_0$ on an arbitrary interval of time $[0,T]$, then $\partial_t\theta\in L^1\left([0,T];L^{\infty}\left(\mathbb{R}^d\right)\right)$, $\theta\in C^1_tC^2_x\left((0,T]\times\mathbb{R}^d\right)$,
\[
\lim_{t\rightarrow0^+}\|\theta(t,\cdot)-\theta_0\|_{W^{1,\infty}}=0,
\]
and
\[
\partial_t\theta(t,x)-\nu\Delta\theta(t,x)=\lambda\left|\nabla\theta(t,x)\right|^p+\mu\left(-\Delta\right)^{\alpha}\theta(t,x),
\]
holds true in the classical (pointwise) sense for every $(t,x)\in(0,T]\times\mathbb{R}^d$. If $[0,T_*)$ is the maximal interval of existence of the strong solution, then we must have
\[
T_*=\sup\left\{T>0:\|\nabla\theta(t,\cdot)\|_{L^{\infty}}<\infty\ \forall t\in[0,T]\right\}.
\]
\end{thm}
\begin{rem}
One can obtain a result analogous to Theorem \ref{shortimeexis} for any $\alpha\in(0,1)$. We restrict ourselves to the case $\alpha\in(0,1/2)$ for the sake of simplicity.	See discussion at the end of \S\ref{nseclable1} for more details.
\end{rem}

Thus, in order to go from local to global well-posedness, it is sufficient to prevent a gradient blowup scenario (in the $L^{\infty}$ norm) in finite time. This will be guaranteed if we impose either a periodicity hypothesis on the initial data or require it to vanish at at infinity, i.e., we further assume that either
\[
\theta_0(x+Le_j)=\theta_0(x), \quad \forall j\in\{1,2,\cdots,d\},\quad x\in\mathbb{R}^d,
\]
where $\{e_j\}_{j=1}^d$ is the standard basis of $\mathbb{R}^d$ and $L>0$, or
\[
\lim_{|x|\rightarrow\infty}|\theta_0(x)|=0.
\]
In this case, we show that the (unique) strong solution arising from such initial data (as defined in Definition \ref{strongsol}) automatically inherits those properties. Moreover, we are able to control its Lipschitz constant for all time by constructing a strong modulus of continuity (Definition \ref{defmod}, below) that must be obeyed by the solution. That is, we establish
\begin{thm}\label{mainthm}
Assume the hypotheses of Theorem \ref{shortimeexis} and suppose further that $\theta_0$ is either periodic with period $L>0$ in every spatial direction or vanishes at infinity. Then there exists a strong solution $\theta$ to \eqref{maineq} on $[0,\infty)$ corresponding to $\theta_0$, which is periodic if $\theta_0$ is (with the same period) or vanishes at infinity if $\theta_0$ does. Furthermore,  $\theta$ is unique in the class of strong solutions and we have the following estimate valid for every $t\geq0$,
\begin{equation}\label{bdlipsch}
\left\|\nabla\theta(t,\cdot)\right\|_{L^{\infty}}\leq Be^{C_0t},
\end{equation}
where $B$ depends only on $\|\theta_0\|_{W^{1,\infty}}$ and $C_0$ depends on $B,\nu,\alpha,d,\mu$, with $C_0$ blowing up as $\alpha\rightarrow 1/2$ or $\nu\rightarrow 0$. In particular, $B$ and $C_0$ do not depend on the period $L$ if $\theta_0$ is periodic, nor on $p$ or $\lambda$.
\end{thm}
\begin{rem}
One can certainly allow for more singular initial data by considering the periodic and whole space scenario separately, and modifying the definition of a strong solution accordingly; see discussion at the end of \S\ref{nseclable1} for more details. Essentially, one only needs to guarantee that the solution immediately experiences parabolic regularity (make sense of the PDE \eqref{maineq} in the pointwise sense on $(0,T]\times\rd$). We chose the space $W^{1,\infty}(\rd)$ and define strong solutions as in Definition \ref{strongsol} in order to handle both scenarios in a simple, unified fashion. That is to say, a direct corollary is that we establish the global well-posedness of regular solutions to the initial value problem associated with \eqref{maineq} when posed with ``periodic boundary conditions''.
\end{rem}

It is unclear at this stage whether the growth in time observed in \eqref{bdlipsch} is sharp or is simply a technical difficulty arising from the proof. Equation \eqref{maineq} does not have any scale invariance, and so our modulus of continuity will be customized for each initial data, complicating the construction. Furthermore, in order to balance out the instabilities arising from the nonlocality without allowing time dependence, the second derivative of the modulus should be bounded from above by a negative constant, a scenario that might lead the modulus of continuity to be negative. This will be made clear at the technical level in \S\ref{proofmainthm}, and touched upon in the conclusion. Moreover, such growth in time is also expected for the linear equation, that is equation \eqref{maineq} with $\lambda=0$.

This paper is organized as follows. In \S2, we list some preliminary estimates and results that will be used later on. Subsection \S 2.1 mainly summarizes the properties and basic results of moduli of continuity used when studying their evolution, most of which are proven in \cite{Kiselev2011,  KNS2008,KNV2007}. In \S2.2, we obtain some pointwise estimates for $(-\Delta)^{\alpha}$. Section 3 deals with the proof of Theorem \ref{shortimeexis}, which mainly follows ideas from \cite{Ben-ArtziAmour1998, BAMJL2000}, slightly modified to take into account nonlocality. Finally, Theorem \ref{mainthm} is proved in \S4, where the modulus is explicitly constructed and shown to be ``preserved'' by the evolution. We conclude with some remarks.

\section{Preliminaries}
In this section, we list some preliminary results and estimates that will be used throughout this work. We summarize the main ingredients introduced in \cite{KNS2008,KNV2007} when studying the evolution of moduli of continuity in \S2.1 . In \S2.2, we obtain some (elementary) pointwise upper bounds for the nonlocal operator $(-\Delta)^{\alpha}$ that we will need in the analysis to follow. In particular, Lemma \ref{fraclapmod} (a generalization of \cite[Proposition~2.5]{Silvestre2007}) is the crucial estimate that will be used to prove the long-time existence of strong solutions, and is the key ingredient that fails when trying to obtain similar results for $\alpha\geq1/2$.

\subsection{Moduli of Continuity.}

\begin{defn}\label{defmod}
We say a function $\omega:[0,\infty)\rightarrow[0,\infty)$ is a modulus of continuity if $\omega\in C([0,\infty))\cap C^2(0,\infty)$, nondecreasing, concave and $\omega(0)=0$. A modulus of continuity $\omega$ is said to be strong if in addition $0<\omega'(0)<\infty$ and $\displaystyle{\lim_{\xi\rightarrow0^+}\omega''(\xi)=-\infty}$.
\end{defn}

\begin{defn}\label{defmod2}
Let $\omega$ be a modulus of continuity. We say a scalar function $\theta\in C(\mathbb{R}^d)$ has modulus of continuity $\omega$ if $|\theta(x)-\theta(y)|\leq\omega(|x-y|)$. We say $\theta$ has strict modulus of continuity $\omega$ if $|\theta(x)-\theta(y)|<\omega(|x-y|)$ whenever $x\neq y$.	
\end{defn}

To avoid cumbersome notation, in the proof of the following two lemmas, we drop the subscript $L^{\infty}$ from $\|\cdot\|_{L^{\infty}}$. Even though they are discussed in \cite{KNS2008,KNV2007}, we prove them again here for the sake of completeness and convenience. Moreover, we find it necessary to rigorously prove Lemma \ref{buildmod}, in order to verify that the control on the Lipschitz constant of the solution is independent of the period length $L>0$ when $\theta_0$ is chosen to be periodic.

\begin{lem}\label{buildmod}
Let $\theta\in W^{1,\infty}\left(\mathbb{R}^d\right)$ be bounded and Lipschitz scalar, and suppose $\omega$ is an unbounded modulus of continuity. Then there exists $B_{\theta}\geq1$ depending only on $\|\theta\|_{L^{\infty}}$ and $\|\nabla\theta\|_{L^{\infty}}$ such that $\theta$ has strict modulus of continuity $\omega(B|x-y|)$ whenever $B\geq B_\theta$.  
\end{lem}
\begin{proof}
	 Chose $B_\theta>0$ such that $\omega(B_\theta)>\max\{2\|\theta\|+1,\|\nabla\theta\|+1\}$, which is possible as $\omega$ is unbounded. As $\omega$ is nondecreasing, we must have $\omega(B)\geq\omega(B_\theta)$ for any $B\geq B_\theta$. Let $\xi:=|x-y|>0$. For $\xi\geq1$, we write:
	\[
	\omega(B\xi)=\omega(B)+\int_B^{B\xi}\omega'(\eta)\ d\eta\geq\omega(B)>2\|\theta\|\geq|\theta(x)-\theta(y)|,
	\]
	meaning $\left|\theta(x)-\theta(y)\right|<\omega(B|x-y|)$ whenever $|x-y|\geq1$. When $\xi\in(0,1)$, we first write  
	\[
	|\theta(x)-\theta(y)|\leq\gradt|x-y|,
	\]
	and note that due to the concavity of $\omega$, the function
	\[
	h(\xi):=\|\nabla\theta\|-\frac{\omega(B\xi)}{\xi},
	\]
	is increasing and so must be negative on $(0,1)$, as $h(1)<0$ by choice of $B$.
\end{proof}

\begin{lem}\label{gradbd}
Suppose $\theta\in C^2(\mathbb{R}^d)$ and has a strong modulus of continuity $\omega$. It then follows that $\theta$ is Lipschitz and $\|\nabla\theta\|_{L^{\infty}}<\omega'(0)$.
\end{lem}
\begin{rem}\label{rmklips}
	That $\theta$ is Lipschitz and $\|\nabla\theta\|_{L^{\infty}}\leq d^{1/2}\omega'(0)$ follows from Definition \ref{defmod2} and the limit definition of a derivative. The important part is the strict inequality, for which we need $\omega''(0)=-\infty$, and $\theta\in C^2$.
\end{rem}
\begin{proof}Let $x^0\in\mathbb{R}^d$ be such that $|\nabla\theta(x^0)|=\gradt$. Let $\xi\in(0,1]$ be arbitrary and let $\displaystyle{y=x^0+\xi e}$, where $e$ is the unit vector in the direction of $\nabla\theta(x^0)$. From the first order Taylor expansion of $\theta$ about $x^0$ we see that 
\[
|\theta(y)-\theta(x^0)|\geq \gradt
\xi-\frac{C\xi^2}{2}\|\nabla^2\theta\|,
\]
here $\|\nabla^2\theta\|$ is just the maximum of all second order derivatives in a ball centered at $x^0$ with radius 1, and $C$ is a combinatorial constant. The left-hand side is at most $\omega(\xi)$, and so after rearranging we get for any $\xi \in(0,1]$,
\begin{equation}\label{comp3}
		\gradt\leq\frac{\omega(\xi)}{\xi}+\frac{C\xi}{2}\|\nabla^2\theta\|.
	\end{equation}
Since $\omega$ is $C^2$ on $(0,\infty)$, and $\displaystyle{\lim_{\xi\rightarrow0^+}\omega''(\xi)=-\infty}$, it follows from the Taylor expansion of $\omega$ around $\xi/2$ that
\[
\omega(\xi)=\omega(\xi/2)+\frac{\omega'(\xi/2)}{2}\xi-\rho(\xi)\xi^2,
\]
where $\displaystyle{\lim_{\xi\rightarrow0^+}\rho(\xi)=\infty}$.
Plugging this into \eqref{comp3} we get
\[
\gradt\leq\frac{\omega(\xi/2)}{\xi}+\frac{\omega'(\xi/2)}{2}+\xi\left(C\|\nabla^2\theta\|-\rho(\xi)\right).
\]
The result now follows by choosing $\xi\in(0,1]$ small enough such that $C\|\nabla^2\theta\|-\rho(\xi)<0$ and noting that 
\[
\frac{\omega(\xi/2)}{\xi}+\frac{\omega'(\xi/2)}{2}<\frac{\omega'(0)}{2}+\frac{\omega'(0)}{2}=\omega'(0),
\]
where we again used the concavity of $\omega$.
\end{proof}

The following lemma is crucial in handling the nonlinear part of the equation, as well as extracting local dissipation from the Laplacian. See \cite[Proposition~2.4]{Kiselev2011} for further insight, and a slightly different proof. We relax the assumptions on the modulus of continuity and only assume it is continuous on $[0,\infty)$, and piecewise $C^2$ on $(0,\infty)$, with finite one-sided derivatives, modulo the condition $\omega''(0)=-\infty$.
\begin{lem}\label{dervh}
	Suppose $\theta$ is $C^2(\mathbb{R}^d)$ and has modulus of continuity $\omega$. If $\theta(x^0)-\theta(y^0)=\omega(|x^0-y^0|)$ for some $x^0\neq y^0$, with $x^0-y^0=(\xi,0,\cdots,0)$, $\xi>0$, then 
	\begin{equation}\label{derv}
	\begin{cases}
		\omega'(\xi^-)\leq\partial_1\theta(x^0)=\partial_1\theta(y^0)\leq\omega'(\xi^+),\\
		\partial_j\theta(x^0)=\partial_j\theta(y^0)=0, \quad j>1
	\end{cases}
	\end{equation}
and 
\begin{equation}\label{lap}
\Delta\theta(x^0)-\Delta\theta(y^0)\leq 4\omega''(\xi^-).
\end{equation}
\end{lem}
\begin{proof}
We start by showing $\partial_j\theta(x^0)=\partial_j\theta(y^0)$ and $\partial^2_j\theta(x^0)-\partial^2_j\theta(y^0)\leq0$. Let $\epsilon>0$ and define:
\begin{align*}
&d_\epsilon^+:=\theta(x^0+\epsilon e_j)-\theta(y^0+\epsilon e_j)-\left[\theta(x^0)-\theta(y^0)\right],\\
&d_\epsilon^-:=\theta(x^0)-\theta(y^0)+\left[\theta(y^0-\epsilon e_j)-\theta(x^0-\epsilon e_j)\right]	,\\
&d_\epsilon:=\left[\theta(x^0+\epsilon e_j)-2\theta(x^0)+\theta(x^0-\epsilon e_j)\right]-\left[\theta(y^0+\epsilon e_j)-2\theta(y^0)+\theta(y^0-\epsilon e_j)\right],
\end{align*}
	where $\{e_j\}_{j=1}^d$ is the standard unit basis of $\mathbb{R}^d$. It is sufficient to show $d_\epsilon^+\leq0$, $d_\epsilon^-\geq0$ and $d_\epsilon\leq0$. But this follows immediately from the fact that $\theta(x^0)-\theta(y^0)=\omega(\xi)$ and $|\theta(x)-\theta(y)|\leq\omega(|x-y|)$ for any $x,y$.  Next, we define 
	\begin{align*}
	&d_{\epsilon,j}^+:=\theta(x^0+\epsilon e_j)-\theta(x^0)=\theta(x^0+\epsilon e_j)-\theta(y^0)-\omega(\xi),\\
	&d_{\epsilon,j}^-:=\theta(x^0)-\theta(x^0-\epsilon e_j)=\omega(\xi)+\theta(y^0)-\theta(x^0-\epsilon e_j).
\end{align*}
Notice that for $j=1$, we have $|x^0+\epsilon e_1-y^0|=\xi+\epsilon$, and $|y^0-x^0+\epsilon e_1|=\xi-\epsilon$ whenever $\epsilon\in(0,\xi/2)$, while for $j>1$, $|x^0+\epsilon e_j-y^0|=|y^0-x^0+\epsilon e_j|=\sqrt{\xi^2+\epsilon^2}$. Hence,
\begin{align}
&d_{\epsilon,j}^+\leq 
\begin{cases}
	\omega(\xi+\epsilon)-\omega(\xi), &j=1,\\
	\omega(\sqrt{\xi^2+\epsilon^2})-\omega(\xi), &j>1
\end{cases},\label{depsplus}\\
&d_{\epsilon,j}^-\geq 
\begin{cases}
	\omega(\xi)-\omega(\xi-\epsilon), &j=1,\\
	\omega(\xi)-\omega(\sqrt{\xi^2+\epsilon^2}), &j>1
\end{cases}\label{depsminus}	,
\end{align}
from which \eqref{derv} follows immediately upon dividing \eqref{depsplus} and \eqref{depsminus} by $\epsilon>0$ and letting $\epsilon\rightarrow 0^+$, since $\omega$ is continuous and have one-sided derivatives. Finally, let $x':=(x^0_2,\cdots,x^0_d)\in\mathbb{R}^{d-1}$ be the other coordinates, and define
 \[
 h(s):=\theta(s,x')-\theta(x^0_1+y^0_1-s,x')-\omega(2s-x^0_1-y^0_1),\ s>\frac{x_1^0+y_1^0}{2}.
 \]
Suppose for the sake of contradiction that $\partial^2_1\theta(x^0)-\partial^2_1\theta(y^0)>4\omega''(\xi^-)$. As $\omega$ is piecewise $C^2$, it follows that there exists some small enough $\epsilon>0$ such that $h(s)$ is $C^2$ on $[x^0_1-\epsilon,x^0_1]$ and $-h''(s)<0$ on that interval. On the one hand, a Lemma of Hopf (or simple calculus) tells us that we must have $h'(x^{0-}_1)>0$. On the other hand, owing to \eqref{derv}, we must have
\[
h'(x^{0-}_1)=2\left(\partial_1\theta(x^0)-\omega'(\xi^-)\right)\leq 2\left(\omega'(\xi^+)-\omega'(\xi^-)\right),
\]
 which leads to a contradiction under the concavity assumption of $\omega$.
\end{proof}
\begin{rem}
Under the concavity assumption of $\omega$, from \eqref{derv} we see that the modulus of continuity cannot be violated at a point where $\omega'$ has a jump discontinuity. 	
\end{rem}

\subsection{Pointwise Estimates for $(-\Delta)^{\alpha}$.}
This subsection is devoted to deriving some simple pointwise upper bounds for the fractional Laplacian. Lemma \ref{bdfraclap} is used in proving local well-posedness in a simple manner, regardless of whether we are in the periodic or whole space setting, while Lemma \ref{decayfraclap} is required when handling the whole space setting. We remark that one can do without Lemma \ref{bdfraclap} by specializing to the periodic or whole space scenario, where short-time existence and regularity can be proven by standard energy techniques and, in the periodic case, by Galerkin approximations. Lemma \ref{bdfraclap} simply allows us to prove local-well-posedness and regularity for either scenario, and arbitrary dimension $d$ in a simple, unified fashion. On the other hand, we emphasize again, that Lemma \ref{fraclapmod} is the key ingredient that allows one to control the nonlocal, destabilizing part, by the local diffusive term, and is the key estimate that is missing when trying to prove similar results when $\alpha\geq1/2$. We remark that the very recent work of Miao and Xue \cite{MX2019} was brought to our attention by one of the anonymous referees. Upon inspection, some version of Lemma \ref{fraclapmod} was proven there for a special class of Fourier multipliers of order strictly less than one. However, the class of operators considered in \cite{MX2019} does not include the fractional Laplacian, since they require the kernel to have a zero average, a property that is not satisfied by the operator $(-\Delta)^{\alpha}$.

Recall the singular integral definition of $(-\Delta)^{\alpha}$
\begin{equation}\label{singint}
(-\Delta)^{\alpha}\theta(x)=C_{d,\alpha}P.V.\int_{\mathbb{R}^d}\frac{\theta(x)-\theta(x-z)}{|z|^{d+2\alpha}}\ dz,
\end{equation}
which is known to be equivalent to the Fourier multiplier definition (in the whole space)
\[
\widehat{(-\Delta)^{\alpha}\theta}(\zeta):=|\zeta|^{2\alpha}\hat\theta(\zeta).
\]We remark that for periodic functions (assume the period is $2\pi$ for simplicity), it is common to instead use the following pointwise formula
\begin{equation}\label{ptwiseper}
(-\Delta)^{\alpha}\theta(x)=C_{d,\alpha}\sum_{k\in\mathbb{Z}^d}\int_{\mathbb{T}^d}\frac{\theta(x)-\theta(x-z)}{|z+k|^{d+2\alpha}}\ dz,
\end{equation}
with \eqref{ptwiseper} known to be equivalent to the (periodic) Fourier multiplier definition
\begin{equation}\label{fourierdef}
(-\Delta)^{\alpha}\theta(x)=\sum_{k\in\mathbb{Z}^d}|k|^{2\alpha}\hat{\theta}(k)e^{i k\cdot x},
\end{equation}
see for instance \cite{CC2004}. Nevertheless, by regularizing the singular integral \eqref{singint} via
\begin{equation}\label{nlabel1}
(-\Delta)^{\alpha}\theta(x)=C_{d,\alpha}\int_{\rd}\frac{2\theta(x)-\theta(x-z)-\theta(x+z)}{|z|^{d+2\alpha}}\ dz,
\end{equation}
which is absolutely convergent for any $\theta\in C^2(\rd)\cap L^{\infty}(\rd)$ and any $\alpha\in(0,1)$, along with using the fact that the function
\[
g(k):=\int_{\rd}\frac{1-\cos(k\cdot z)}{|z|^{d+2\alpha}}\ dz,\quad k\in\rd,
\]
is rotation invariant, allows us to easily establish the equivalence of \eqref{singint} and \eqref{fourierdef} (and hence to \eqref{ptwiseper}) in case $\theta$ happens to be periodic, by appropriately choosing the normalizing constant $C_{d,\alpha}$. We prefer to work with the representation \eqref{singint}, as it allows us to easily obtain the required bounds and continuity estimates, regardless of whether the function is periodic or not.
\begin{lem}\label{bdfraclap}
	Let $\alpha\in(0,1/2)$, $\gamma\in(2\alpha,1]$, $\theta\in L^{\infty}(\rd)\cap C^{0,\gamma}\left(\mathbb{R}^d\right)$ and define
	\[
	[\theta]_{C^{0,\gamma}}:=\sup_{x\neq y}\frac{|\theta(x)-\theta(y)|}{|x-y|^{\gamma}}.
	\] 
	It follows that $(-\Delta)^{\alpha}\theta\in L^{\infty}(\mathbb{R}^d)$ and
	\begin{equation}\label{nlabel2}
	\|(-\Delta)^{\alpha}\theta\|_{L^{\infty}}\leq \frac{\gamma C_{d,\alpha}|\mathbb{S}^{d-1}|}{\alpha(\gamma-2\alpha)}\left\|\theta\right\|_{L^{\infty}}^{1-\frac{2\alpha}{\gamma}}[\theta]_{C^{0,\gamma}}^{\frac{2\alpha}{\gamma}}.
	\end{equation}
	Similarly, if $\alpha\in[1/2,1)$, $\gamma\in (2\alpha-1,1]$, and $\theta\in L^{\infty}(\rd)\cap C^{1,\gamma}\left(\mathbb{R}^d\right)$ we must have $(-\Delta)^{\alpha}\theta\in L^{\infty}(\rd)$ and 
	\begin{equation}\label{nlabel3}
	\|(-\Delta)^{\alpha}\theta\|_{L^{\infty}}\leq \frac{(1+\gamma)C_{d,\alpha}|\mathbb{S}^{d-1}|}{\alpha(1+\gamma-2\alpha)}\left\|\theta\right\|_{L^{\infty}}^{1-\frac{2\alpha}{1+\gamma}}[\nabla\theta]_{C^{0,\gamma}}^{\frac{2\alpha}{1+\gamma}}.
	\end{equation}
\end{lem}
\begin{proof}
For $\alpha\in(0,1/2)$, the singular integral \eqref{singint} is absolutely convergent when $\theta\in L^{\infty}(\rd)\cap C^{0,\gamma}\left(\mathbb{R}^d\right)$, $\beta\in(2\alpha,1]$. Moreover, if $\theta$ is constant, the result is trivial, so we assume otherwise. For fixed $R>0$, we have
\begin{align*}
\left|(-\Delta)^{\alpha}\theta(x)\right|&\leq C_{d,\alpha}\int_{|z|\leq R}\frac{|\theta(x)-\theta(x-z)|}{|z|^{d+2\alpha}}dz+C_{d,\alpha}\int_{|z|> R}\frac{|\theta(x)-\theta(x-z)|}{|z|^{d+2\alpha}}dz\\
&\leq 2C_{d,\alpha}|\mathbb{S}^{d-1}|\left([\theta]_{C^{0,\gamma}}\int_0^R\rho^{\gamma-2\alpha-1}\ d\rho+\|\theta\|_{L^{\infty}}\int_R^{\infty}\rho^{-2\alpha-1}\ d\rho\right)\\
&\leq 2C_{d,\alpha}|\mathbb{S}^{d-1}|\left(\frac{R^{\gamma-2\alpha}}{\gamma-2\alpha}[\theta]_{C^{0,\gamma}}+\frac{R^{-2\alpha}}{2\alpha}\left\|\theta\right\|_{L^{\infty}}\right).
\end{align*}
Bound \eqref{nlabel2} now follows by choosing $R:=\left(\|\theta\|_{L^{\infty}}[\theta]_{C^{0,\gamma}}^{-1}\right)^{1/\gamma}$. When $\alpha\in[1/2,1)$ we use the mean value theorem to get 
\[
|2\theta(x)-\theta(x-z)-\theta(x+z)|\leq C_d[\nabla\theta]_{C^{0,\gamma}}|z|^{1+\gamma},
\]
and so if $\theta\in L^{\infty}(\rd)\cap C^{1,\gamma}\left(\mathbb{R}^d\right)$ with $\gamma\in(2\alpha-1,1]$, we can use the regularization \eqref{nlabel1} to obtain 
\begin{align*}
\left|(-\Delta)^{\alpha}\theta(x)\right|&\leq C_{d,\alpha}|\mathbb{S}^{d-1}|\left([\nabla\theta]_{C^{0,\gamma}}\int_0^R\rho^{\gamma-2\alpha}\ d\rho+\|\theta\|_{L^{\infty}}\int_R^{\infty}\rho^{-2\alpha-1}\ d\rho\right)\\
&\leq C_{d,\alpha}|\mathbb{S}^{d-1}|\left(\frac{R^{\gamma+1-2\alpha}}{\gamma+1-2\alpha}[\nabla\theta]_{C^{0,\gamma}}+\frac{R^{-2\alpha}}{2\alpha}\left\|\theta\right\|_{L^{\infty}}\right).
\end{align*}
We conclude by choosing $R:=\left(\|\theta\|_{L^{\infty}}[\nabla\theta]_{C^{0,\gamma}}^{-1}\right)^{1/(1+\gamma)}$.
\end{proof}
\begin{lem}\label{decayfraclap}
For integer $k\geq0$, denote by $C^k_0(\mathbb{R}^d)\subset W^{k,\infty}(\mathbb{R}^d)$ the space of all $C^k(\mathbb{R}^d)$ functions such that all derivatives up to order $k$ are bounded and vanish at infinity, i.e.,
\[
\lim_{|x|\rightarrow\infty}|D^{\beta}\theta(x)|=0,\quad \forall |\beta|\leq k.
\]
If $\alpha\in(0,1/2)$, then $(-\Delta)^{\alpha}\theta\in C^{k-1}_0(\mathbb{R}^d)$, whenever $k\geq1$. If $\alpha\in[1/2,1)$, then $(-\Delta)^{\alpha}\theta\in C^{k-2}_0(\mathbb{R}^d)$, whenever $k\geq2$.
\end{lem}
\begin{proof}
It suffices to prove the results for $k=1,2$, when $\alpha\in(0,1/2)$ and $[1/2,1)$, respectively. For $\alpha\in(0,1)$, we regularize the singular integral \eqref{singint} by
 \[
 (-\Delta)^{\alpha}\theta(x)=C_{d,\alpha}\int_{\mathbb{R}^d}\frac{\theta(x)-\theta(x-y)-y\cdot \nabla\theta(x)\chi_{|y|\leq 1}(y)}{|y|^{d+2\alpha}}\ dy,
 \]
making the above integral absolutely convergent for $\theta\in C^1$, if $\alpha\in(0,1/2)$ and $\theta\in C^2$, if $\alpha\in[1/2,1)$. We start by splitting the integral into a singular part, intermediate part and decaying part as follows
\begin{align*}
&I_S:=\int_{|y|\leq1}\frac{\theta(x)-\theta(x-y)-y\cdot \nabla\theta(x)}{|y|^{d+2\alpha}}\ dy,\\
&I_M:=\int_{1\leq|y|\leq R}\frac{\theta(x)-\theta(x-y)}{|y|^{d+2\alpha}}\ dy,\\
&I_R:=\int_{|y|\geq R}\frac{\theta(x)-\theta(x-y)}{|y|^{d+2\alpha}}\ dy.
\end{align*}
 In what follows, $C_{d,\alpha}$ always denotes a positive constant depending on $d,\alpha$, degenerating as $\alpha\rightarrow0^+$ or $1^-$, and whose value may change from line to line. For any given $\epsilon>0$, we start by choosing a large enough $R>1$ such that 
 \[
 \int_{|y|\geq R}|y|^{-d-2\alpha}\ dz<\frac{\epsilon}{6C_{d,\alpha}\|\theta\|_{L^{\infty}}},
 \]
 making $I_R<\epsilon/3$. Next, we chose a large enough $N_0>R$ such that 
 \[
 |\theta(z)|\leq\frac{\epsilon}{6 C_{d,\alpha}},
 \] 
 whenever $|z|\geq N_0-R$, rendering $I_M<\epsilon/3$ provided $|x|>N_0$. To handle $I_S$, notice that  given $(x,y)\in\mathbb{R}^d\times\mathbb{R}^d$, by the mean value theorem, we can find some $\lambda=\lambda(x,y)\in(0,1)$ such that 
 \[
 \theta(x)-\theta(x-y)=y\cdot\nabla\theta\left(x+(\lambda-1)y\right),
 \]
implying the singular integrand of $I_S$ is bounded from above by  
\begin{equation}\label{mvt}
\frac{|\theta(x)-\theta(x-y)-y\cdot\nabla\theta(x)|}{|y|^{d+2\alpha}}\leq\frac{\left|\nabla\theta\left(x+(\lambda-1)y\right)-\nabla\theta(x)\right|}{|y|^{d+2\alpha-1}}.
\end{equation}
For $\alpha\in(0,1/2)$, we can chose a large enough $N_1>1$ such that, whenever $|z|\geq N_1-1$, 
\[
\left|\nabla\theta(z)\right|<\frac{\epsilon}{6C_{d,\alpha}},
\]
making $I_S<\epsilon/3$ when $|x|\geq N_1$. This concludes the case when $\alpha\in(0,1/2)$. For $\alpha\in[1/2,1)$, we apply the mean value theorem once again to \eqref{mvt} to get a $\sigma\in(0,1)$ and conclude that the singular integrand is now dominated by
\[
|y|^{2-2\alpha-d}\left|\nabla^2\theta(x+\sigma(\lambda-1)y)\right|,
\]
allowing us to conclude by choosing a large enough $N_1$ such that
\[
\left|\nabla^2\theta(z)\right|<\frac{\epsilon}{3C_{d,\alpha}},
\]
whenever $|z|\geq N_1-1$, meaning $I_S<\epsilon/3$ when $|x|\geq N_1$.
\end{proof}	
\begin{lem}\label{fraclapmod}
 Suppose $\theta\in C(\mathbb{R}^d)$ has a strong modulus of continuity $\omega$, and let $\alpha\in(0,1/2)$. Then $(-\Delta)^{\alpha}\theta$ has modulus of continuity
 \begin{equation}\label{modlap}
 \widetilde{\omega}(\xi):=C_{d,\alpha}|\mathbb{S}^{d-1}|\alpha^{-1}\int_0^\xi\frac{\omega'(\eta)}{\eta^{2\alpha}}\ d\eta.
 \end{equation}
 \end{lem}
 \begin{rem}
 The modulus of continuity $\omega$ need not be strong. All what is required is for the integral \eqref{modlap} to be convergent, that is we require $\omega(\xi)= O(\xi^{\beta})$ some $\beta\in(2\alpha,1]$ when $\xi$ is small.
 \end{rem}

 \begin{proof}
Following Remark \ref{rmklips}, we must have $\theta\in W^{1,\infty}(\mathbb{R}^d)$, and so for $\alpha\in(0,1/2)$, the singular integral \eqref{singint} is absolutely convergent. Therefore, for arbitrary $\rho>0$, $(x,z)\in\mathbb{R}^d\times\mathbb{R}^d$, we must have
 \[
 \left|(-\Delta)^{\alpha}\theta(x)-(-\Delta)^{\alpha}\theta(z)\right|\leq C_{d,\alpha}(I_1+I_2),
 \]
 where
 \begin{align*}
 &I_1:=\left|\int_{|y|\leq\rho}\frac{\theta(x)-\theta(x-y)-\left(\theta(z)-\theta(z-y)\right)}{|y|^{d+2\alpha}}\ dy\right|,\\
 &I_2:=\left|\int_{|y|>\rho}\frac{\theta(x)-\theta(z)-\left(\theta(x-y)-\theta(z-y)\right)}{|y|^{d+2\alpha}}\ dy\right|	.
 \end{align*}
For $I_1$, we estimate from above by
\[
I_1\leq  2\int_{|y|\leq\rho}\frac{\omega(|y|)}{|y|^{d+2\alpha}}\ dy=2|\mathbb{S}^{d-1}|\int_0^\rho\frac{\omega(\eta)}{\eta^{2\alpha+1}}\ d\eta=|\mathbb{S}^{d-1}|\alpha^{-1}\int_0^\rho\frac{\omega'(\eta)}{\eta^{2\alpha}}\ d\eta-|\mathbb{S}^{d-1}|\frac{\omega(\rho)}{\alpha\rho^{2\alpha}}
\]
where we integrated by parts in the last step. For $I_2$, we have
\[
I_2\leq2\omega(|x-z|)\int_{|y|\geq \rho}|y|^{-d-2\alpha}\ dy=|\mathbb{S}^{d-1}|\alpha^{-1}\frac{\omega(|x-z|)}{\rho^{2\alpha}},
\]
from which we conclude by choosing $\rho=|x-z|$.
  \end{proof}
\section{Proof of Theorem \ref{shortimeexis}}\label{nseclable1}
In this section, $\nabla$ always denotes the gradient vector acting on spatial coordinates, while $C_{d,\alpha}\geq1$ always denotes an absolute constant depending on the dimension $d$ and $\alpha$, may blow up as $\alpha\rightarrow1/2$, and whose value may change from line to line. Let us start by recalling some properties of the heat kernel
\begin{align}
&\int_{\mathbb{R}^d}\Psi(s,y)\ dy=1,\label{intkern}\\
&\int_{\mathbb{R}^d}\left|\nabla\Psi(s,x-y)\right|\ dy=\frac{C_d}{\sqrt{\nu s}},\label{intgradkern}\\
&\int_{\mathbb{R}^d}|x-y|^{\gamma}|\partial_s\Psi(s,x-y)|\ dy\leq C_d\nu^{\gamma/2}s^{\gamma/2-1},\label{intdtkern}\\
&\int_{\mathbb{R}^d}\left|\nabla\Psi(s,x-y)-\nabla\Psi(s,z-y)\right|\ dy\leq \frac{C_{d}}{\nu s}|x-z|,	\label{lipkern}
\end{align}
where $s,\gamma>0$, and $(x,z)\in\mathbb{R}^d\times\mathbb{R}^d$ are arbitrary. From \eqref{intgradkern} and \eqref{lipkern} we get
\begin{equation}\label{holdkern}
	\int_{\mathbb{R}^d}\left|\nabla\Psi(s,x-y)-\nabla\Psi(s,z-y)\right|\ dy\leq \frac{C_{d,\beta}|x-z|^{\beta}}{(\nu s)^{\frac{1}{2}(1+\beta)}},
\end{equation}
where $\beta\in(0,1)$ is arbitrary. Properties \eqref{intkern}-\eqref{intdtkern} follow by explicit calculations, while it is somewhat tedious (yet straightforward) to prove inequality \eqref{lipkern}, see for instance \cite[Lemma~4.3]{BAMJL2000}. We will also make use of the following Gronwall-type inequality, which can be proved by first using the H\"older inequality and then proceeding as in the proof of the integral version of Gronwall's inequality \cite[Appendix B]{EvansPDE2010}.
\begin{lem}\label{gronwall}
Let $q\in[1,\infty)$, $1/q+1/r=1$, $T_2\geq T_1$, $C_0\geq0$ and assume that $g\in L^q(T_1,T_2)$, $f\in L^r(0,T_2-T_1)$ are both non-negative. If 
\[
g(t)\leq \int_{T_1}^{t}f(t-s)g(s)\ ds+C_0,\quad a.e.\  t\in[T_1,T_2],
\]
then
\[
g(t)\leq C_0\left[2\left(\int_{0}^{t-T_1}|f(s)|^rds\right)^{1/r}\left(\int_{T_1}^te^{h(t)-h(s)}ds\right)^{1/q}+1\right],\quad a.e.\ t\in[T_1,T_2],
\]
where
\[
h(t):=2^q\int_{T_1}^t\left(\int_{0}^{s-T_1}|f(\sigma)|^rd\sigma\right)^{q/r}ds.
\]
\end{lem}
The proof of Theorem \ref{shortimeexis} closely follows the ideas presented in \cite{Ben-ArtziAmour1998, BAMJL2000}, and will be broken down into several propositions. We begin by constructing strong solutions that exist at least for a short time and which inherit periodicity and decay properties from the initial data, Proposition \ref{construct}. This is followed by proving that strong solutions depend continuously on initial data and hence are unique in their own class, Proposition \ref{wellposedness}. Those two propositions  give us a local well-posedness result in the space $W^{1,\infty}(\rd)\cap X$, where $X$ is either the space of continuous periodic functions defined on $\rd$ or the space of functions that vanish at infinity. We conclude by showing that such solutions experience parabolic regularity (that is, they are classical) and derive a regularity criterion in Propositions \ref{parabolic} and \ref{criteria}, respectively.

\begin{prop}\label{construct}
	Let $d\in \mathbb{N}$, $\nu>0$, $\alpha\in(0,1/2)$, $\lambda\in\mathbb{R}$, $\mu>0$, $p\in[1,\infty)$ and $\theta_0\in W^{1,\infty}\left(\mathbb{R}^d\right)$ with no further restrictions. Then there is a $T_0=T_0(\theta_0,d,\alpha,p,\nu,\mu,\lambda)>0$ and a strong solution $\theta$ to \eqref{maineq} on $[0,T_0]$ corresponding to $\theta_0$. Furthermore, if $\theta_0$ is periodic with period $L>0$, then so is $\theta(t,\cdot)$, and if $\theta_0\in C_0(\mathbb{R}^d)$, then so is $\theta(t,\cdot)$ for $t\in[0,T_0]$.
\end{prop}
\begin{proof}
For $T>0$, let $X_T$ be the Banach space $X_T:=C\left([0,T];W^{1,\infty}(\mathbb{R}^d)\right)$ with the norm
\[
\|f\|_{X_T}:=\max_{t\in[0,T]}\|f(t)\|_{W^{1,\infty}}.
\]
We will construct a strong solution by choosing a small enough $T_0>0$ such that the inductively defined sequence of functions $\left\{\theta_k\right\}_{k=1}^{\infty}$
\begin{align*}
\theta_{1}(t,x)&:=\int_{\mathbb{R}^d}\Psi(t,y)\theta_0(x-y)\ dy,\\
\theta_{k}(t,x)&:=\theta_{1}(t,x)+\lambda\int_0^t\int_{\mathbb{R}^d}\Psi(t-s,x-y)\left|\nabla\theta_{k-1}(s,y)\right|^p\ dy\ ds\\
&\qquad\quad\quad +\mu\int_0^t\int_{\mathbb{R}^d}\Psi(t-s,x-y)\left(-\Delta\right)^{\alpha}\theta_{k-1}(s,y)\ dy\ ds, \quad k\geq2
\end{align*}
is Cauchy in $X_{T_0}$. We start by obtaining some uniform bounds. Let 
\[
M_0:=1+\|\theta_0\|_{L^{\infty}}, \quad M_1:=1+\|\nabla\theta_0\|_{L^{\infty}},
\]
\[
\kappa_0:=C_{d,\alpha}\left(2p|\lambda|M_1^p+\mu M_0^{1-2\alpha}M_1^{2\alpha}\right),
\]
and set
\[
T_0:=\frac{1}{16}\min\left\{\nu\kappa_0^{-2},\kappa_0^{-1}\right\}>0.
\]
We obtain the following bounds, uniform in $k\in\mathbb{N}$, $t\in[0,T_0]$,
\begin{align}
&\|\theta_k(t,\cdot)\|_{L^{\infty}}\leq M_0, \quad \|\nabla\theta_k(t,\cdot)\|_{L^{\infty}}\leq M_1,\label{supbds}
\end{align}
via an inductive argument: they hold trivially for $\theta_1$, and assuming they are true for $\theta_{k-1}$, we get, by using \eqref{intkern}, along with bound \eqref{nlabel2} with $\gamma=1$ from Lemma \ref{bdfraclap}, that
\begin{align*}
|\theta_{k}(t,x)|\leq & \|\theta_{0}\|_{L^{\infty}}+|\lambda|\int_0^t\left\|\nabla\theta_{k-1}(s,\cdot)\right\|_{L^{\infty}}^pds\\
&+C_{d,\alpha}\mu\int_0^t\left\|\theta_{k-1}(s,\cdot)\right\|_{L^{\infty}}^{1-2\alpha}\left\|\nabla\theta_{k-1}(s,\cdot)\right\|_{L^{\infty}}^{2\alpha}ds\\
&\leq \|\theta_{0}\|_{L^{\infty}}+\left(|\lambda|M_1^p+\mu C_{d,\alpha}M_0^{1-2\alpha}M_1^{2\alpha}\right)T_0\leq\|\theta_{0}\|_{L^{\infty}}+\kappa_0T_0,
\end{align*}
By choice of $T_0$, the right-hand side is bounded from above by $M_0$. Similarly, except now using \eqref{intgradkern}, we get that
\begin{align*}
|\nabla\theta_k(t,x)|\leq &\left\|\nabla\theta_0\right\|_{L^{\infty}}+C_{d,\alpha}\nu^{-1/2}\left(|\lambda|M_1^p+\mu M_0^{1-2\alpha}M_1^{2\alpha}\right)\int_0^t s^{-1/2}ds\\
&\leq \left\|\nabla\theta_0\right\|_{L^{\infty}}+\sqrt{\frac{\kappa_0^2T_0}{\nu}},
\end{align*}
and the right-hand side is bounded by $M_1$ by choice of $T_0$, closing the inductive argument. To show that the sequence is Cauchy in $X_{T_0}$, it is sufficient to show that
\begin{equation}\label{cauchy}
\|\theta_k-\theta_{k-1}\|_{X_{T_0}}\leq \frac{1}{2^{k-1}},\quad k\geq2.
\end{equation}
To begin, notice that by choice of $T_0$ and bound \eqref{nlabel2} with $\gamma=1$ from Lemma \ref{bdfraclap}, we have, whenever $(t,x)\in[0,T_0]\times\mathbb{R}^d$, 
\begin{align*}
|\theta_2(t,x)-\theta_1(t,x)|&\leq |\lambda|M_1^pT_0+\mu C_{d,\alpha}M_0^{1-2\alpha}M_1^{2\alpha}T_0\leq \kappa_0T_0\leq \frac{1}{4},\\
|\nabla\theta_2(t,x)-\nabla\theta_1(t,x)|&\leq C_d|\lambda|M_1^p\sqrt{\frac{T_0}{\nu}}+\mu C_{d,\alpha}M_0^{1-2\alpha}M_1^{2\alpha}\sqrt{\frac{T_0}{\nu}}\leq\sqrt{\frac{\kappa_0^2T_0}{\nu}}\leq \frac{1}{4},
\end{align*}
meaning $\|\theta_2-\theta_1\|_{X_{T_0}}\leq1/2$. As $|a^p-b^p|\leq p|a-b|(a^{p-1}+b^{p-1})$, for $p\geq1$, $a,b\geq0$, similar calculations yield, whenever $k\geq 3$ and $(t,x)\in[0,T_0]\times\mathbb{R}^d$,
\begin{align*}
|\theta_{k}(t,x)-\theta_{k-1}(t,x)|&\leq \left(2p|\lambda|M_1^{p-1}+\mu C_{d,\alpha}\right)T_0\|\theta_{k-1}-\theta_{k-2}\|_{X_{T_0}}\\
&\leq\kappa_0T_0\|\theta_{k-1}-\theta_{k-2}\|_{X_{T_0}}\leq\frac{1}{4}\|\theta_{k-1}-\theta_{k-2}\|_{X_{T_0}},\\
|\nabla\theta_{k}(t,x)-\nabla\theta_{k-1}(t,x)|&\leq\left(2C_dp|\lambda|M_1^{p-1}+\mu C_{d,\alpha}\right)\sqrt{\frac{T_0}{\nu}}\|\theta_{k-1}-\theta_{k-2}\|_{X_{T_0}},\\
&\leq\sqrt{\frac{\kappa_0^2T_0}{\nu}}\|\theta_{k-1}-\theta_{k-2}\|_{X_{T_0}}\leq\frac{1}{4}\|\theta_{k-1}-\theta_{k-2}\|_{X_{T_0}},
\end{align*}
meaning, 
\[
\|\theta_{k}-\theta_{k-1}\|_{X_{T_0}}\leq\frac{1}{2}\|\theta_{k-1}-\theta_{k-2}\|_{X_{T_0}},\quad k\geq3,
\]
making \eqref{cauchy} true. It follows that $\{\theta_k\}_{k=1}^{\infty}$ converges to some $\theta$ in the norm topology of $X_{T_0}$ and so, by utilizing Lemma \ref{bdfraclap} one more time,
\begin{align*}
\theta(t,x)=&\int_{\mathbb{R}^d}\Psi(t,y)\theta_0(x-y)\ dy+\lambda\int_0^t\int_{\mathbb{R}^d}\Psi(t-s,x-y)\left|\nabla\theta(s,y)\right|^p\ dy\ ds\\
&+\mu\int_0^t\int_{\mathbb{R}^d}\Psi(t-s,x-y)\left(-\Delta\right)^{\alpha}\theta(s,y)\ dy\ ds,\quad (t,x)\in[0,T_0]\times\mathbb{R}^d
\end{align*}
meaning $\theta$ is a strong solution on $[0,T_0]$ corresponding to $\theta_0$, with the extra regularity $C_tC_x^1\left((0,T_0]\times\mathbb{R}^d\right)$. 

It is clear that if $\theta_0$ is periodic with period $L>0$, then so is each $\theta_k(t,\cdot)$, and so the same can be said of the limiting function. We now argue that if $\theta_0\in C_0(\mathbb{R}^d)$, then $\theta_k(t,\cdot)\in C_0^1(\mathbb{R}^d)$ for each fixed $t>0$. Since
\begin{align*}
&|\Psi(t,y)\theta_0(x-y)|\leq \|\theta_0\|_{L^{\infty}}\Psi(t,y)\in L^1(\mathbb{R}^d),\\
&|\nabla\Psi(t,y)\theta_0(x-y)|\leq \|\theta_0\|_{L^{\infty}}|\nabla\Psi(t,y)|\in L^1(\mathbb{R}^d),	
\end{align*}
hold uniformly in $x\in\mathbb{R}^d$, we conclude that $\theta_1(t,\cdot)\in C^1_0(\mathbb{R}^d)$ for $t>0$. Assuming $\theta_{k-1}(t,\cdot)\in C_0^1(\mathbb{R}^d)$, by virtue of the following bounds holding uniformly in $x\in\mathbb{R}^d$, $s\in[0,t]$,
\begin{align*}
&|\Psi(s,y)||\nabla\theta_{k-1}(t-s,x-y)|^p\leq \Psi(s,y)M_1^p\in L^1\left([0,t]\times\mathbb{R}^d\right),\\
&|\nabla\Psi(s,y)||\nabla\theta_{k-1}(t-s,x-y)|^p\leq |\nabla\Psi(s,y)|M_1^p\in L^1\left([0,t]\times\mathbb{R}^d\right),
\end{align*}
we get that
\begin{align*}
&\lim_{|x|\rightarrow\infty}\left|\int_0^t\int_{\mathbb{R}^d}\Psi(s,y)|\nabla\theta_{k-1}(t-s,x-y)|^p\ dy\ ds\right|\\
&=\lim_{|x|\rightarrow\infty}\left|\int_0^t\int_{\mathbb{R}^d}\nabla\Psi(s,y)|\nabla\theta_{k-1}(t-s,x-y)|^p\ dy\ ds\right|=0.
\end{align*}
Similarly, utilizing Lemmas \ref{bdfraclap} and \ref{decayfraclap}, we conclude that
\begin{align*}
&\lim_{|x|\rightarrow\infty}\left|\int_0^t\int_{\mathbb{R}^d}\Psi(s,y)(-\Delta)^{\alpha}\theta_{k-1}(t-s,x-y)\ dy\ ds\right|\\
&=\lim_{|x|\rightarrow\infty}\left|\int_0^t\int_{\mathbb{R}^d}\nabla\Psi(s,y)(-\Delta)^{\alpha}\theta_{k-1}(t-s,x-y)\ dy\ ds\right|=0,
\end{align*}
meaning $\theta_k(t,\cdot)\in C_0^1(\mathbb{R}^d)$ for every $t>0$. By virtue of the convergence in the norm topology of $X_{T_0}$, we must have $\theta(t,\cdot)\in C_0^1\left(\mathbb{R}^d\right)$ whenever $t\in(0,T_0]$.
\end{proof}
\begin{prop}\label{wellposedness}
Let $T_2\geq T_1$, $\theta_0\in W^{1,\infty}(\mathbb{R}^d)$, and suppose $\theta$ is a strong solution corresponding to $\theta_0$ on $[T_1,T_2]$. It follows that
\[
\lim_{t\rightarrow T_1^+}\|\theta(t,\cdot)-\theta_0\|_{W^{1,\infty}}=0.
\]
Furthermore, if $\theta_1\in W^{1,\infty}\left(\mathbb{R}^d\right)$ and $\varphi$ is a strong solution corresponding to $\theta_1$ on $[T_1,T_2]$, then
\[
\left\|\theta(t,\cdot)-\varphi(t,\cdot)\right\|_{W^{1,\infty}}\leq\left\|\theta_0-\theta_1\right\|_{W^{1,\infty}}\gamma(t),\quad t\in[T_1,T_2],
\]
where $\gamma\in C[T_1,T_2]$ is a positive, increasing function depending on $\alpha,p,\nu, \lambda,\mu$ and the $L^{\infty}\left([T_1,T_2];W^{1,\infty}\left(\mathbb{R}^d\right)\right)$ norms of $\theta$ and $\varphi$.
\end{prop}
\begin{proof}
From the uniform continuity of $\theta_0$, it is clear that 
\begin{equation}\label{convtheta}
\lim_{t\rightarrow T_1^+}\|\theta(t,\cdot)-\theta_0\|_{L^{\infty}}=0,
\end{equation}
and so it remains to show that 
\[
\lim_{t\rightarrow T_1^+}\|\nabla\theta(t,\cdot)-\nabla\theta_0\|_{L^{\infty}}=0.
\]
To do so, first of all notice that as $\theta\in C([T_1,T_2];W^{1,\infty}(\mathbb{R}^d))$, $\nabla\theta(t,x)$ converges to some vector $g(x)$ as $t\rightarrow T_1^+$ in the norm topology of $L^{\infty}(\mathbb{R}^d)$, and all what is needed is to show that $g(x)=\nabla\theta_0(x)$ for almost every $x\in\mathbb{R}^d$. This can be done as follows: let $x_0\in\mathbb{R}^d$ and $R>0$ be arbitrary, and let $\chi_R$ be a smooth function compactly supported in a ball of radius $R$ centered at $x_0$. Then we must have 
\begin{align*}
&\left|\int_{|x-x_0|\leq R}\left(g(x)-\nabla\theta_0(x)\right)\chi_R(x)dx\right|=\lim_{t\rightarrow T_1^+}\left|\int_{|x-x_0|\leq R}\left(\nabla\theta(t,x)-\nabla\theta_0(x)\right)\chi_R(x)dx\right|\\
&=\lim_{t\rightarrow T_1^+}\left|\int_{|x-x_0|\leq R}\left(\theta(t,x)-\theta_0(x)\right)\nabla\chi_R(x)dx\right|=0,
\end{align*}
owing to \eqref{convtheta}, and the fact that $\chi_R$ is compactly supported.

Now, let $w(t,x):=\theta(t,x)-\varphi(t,x)$, and notice that
\begin{align}\label{diffeq}
w(t,x)=&\int_{\mathbb{R}^d}\Psi(t-T_1,y)w_0(x-y)dy+\mu\int_{T_1}^t\int_{\mathbb{R}^d}\Psi(t-s,x-y)\left(-\Delta\right)^{\alpha}w(s,y)dyds\nonumber\\
&+\lambda\int_{T_1}^t\int_{\mathbb{R}^d}\Psi(t-s,x-y)\left(\left|\nabla\theta(s,y)\right|^p-\left|\nabla\varphi(s,y)\right|^p\right)dyds.
\end{align}
As $|a^p-b^p|\leq p|a-b|(a^{p-1}+b^{p-1})$ when $p\in[1,\infty)$, we see that
\begin{align}\label{estpoint}
|w(t,x)|\leq& \|w_0\|_{L^{\infty}}+\mu C_{d,\alpha}\int_{T_1}^t\left\|w(s,\cdot)\right\|_{L^{\infty}}^{1-2\alpha}\left\|\nabla w(s,\cdot)\right\|^{2\alpha}_{L^{\infty}}\ ds\nonumber\\
&+p|\lambda|\max_{s\in[T_1,T_2]}\left[\|\nabla\theta(s,\cdot)\|^{p-1}_{L^{\infty}}+\|\nabla\varphi(s,\cdot)\|^{p-1}_{L^{\infty}}\right]\int_{T_1}^t\left\|\nabla w(s,\cdot)\right\|_{L^{\infty}}ds\nonumber\\
\leq &\|w_0\|_{L^{\infty}}+A\int_{T_1}^t\|w(s,\cdot)\|_{W^{1,\infty}}ds,
\end{align}
where 
\[
A:=2\left(\mu C_{d,\alpha} +p|\lambda|\max_{s\in[T_1,T_2]}\left(\left\|\nabla \theta(s,\cdot)\right\|^{p-1}_{L^{\infty}}+\left\|\nabla \varphi(s,\cdot)\right\|^{p-1}_{L^{\infty}}\right)\right).
\]
Similarly, applying $\nabla$ to \eqref{diffeq} and using \eqref{intgradkern}, we obtain
\begin{align}\label{estgrad}
|\nabla w(t,x)|\leq& \|\nabla w_0\|_{L^{\infty}}+\frac{A}{2\sqrt\nu}\int_{T_1}^t(t-s)^{-1/2}\left\|w(s,\cdot)\right\|_{L^{\infty}}^{1-2\alpha}\left\|\nabla w(s,\cdot)\right\|^{2\alpha}_{L^{\infty}}ds\nonumber\\
&+\frac{A}{2\sqrt\nu}\int_{T_1}^t(t-s)^{-1/2}\left\|\nabla w(s,\cdot)\right\|_{L^{\infty}}ds\nonumber\\
&\leq \|\nabla w_0\|_{L^{\infty}}+A\nu^{-1/2}\int_{T_1}^t(t-s)^{-1/2}\left\| w(s,\cdot)\right\|_{W^{1,\infty}}ds.	
\end{align}
Adding inequalities \eqref{estpoint} and \eqref{estgrad}, while setting $g(t):=\left\|w(t,\cdot)\right\|_{W^{1,\infty}\left(\mathbb{R}^d\right)}$ and $f(\sigma):=A\left((\nu \sigma)^{-1/2}+1\right)$ we  obtain, for every $t\in[T_1,T_2]$,
\[
g(t)\leq \int_{T_1}^tf(t-s)g(s)\ ds+\left\|w_0\right\|_{W^{1,\infty}}.
\]
The result now follows from Lemma \ref{gronwall}, by choosing, for instance, $q=3,r=3/2$.
\end{proof}
\begin{rem}\label{imprmk}
A direct consequence of Propositions \ref{construct} and \ref{wellposedness} is that if $\theta_0$ is periodic or vanishes at infinity, and if $\theta$ is the strong solution corresponding to $\theta_0$ on $[T_1,T_2]$, then $\theta(t,\cdot)$ is periodic or vanishes at infinity if $\theta_0$ is. 	
\end{rem}

To avoid cumbersome notation, in Proposition \ref{parabolic} we work with strong solutions posed on $[0,T]$, without any loss in generality. Further, we write $a\lesssim b$ whenever there exists a constant $C>0$, depending (possibly nonlinearly) on $d,\alpha,p,\beta,\nu,\mu,\lambda$ and 
\[
\max_{t\in[0,T]}\|\theta(t,\cdot)\|_{W^{1,\infty}},
\]
with $\beta\in(0,1)$ such that $a\leq Cb$ uniformly in $t\in[0,T]$. This notation is only used in the proof of Proposition \ref{parabolic}.
\begin{prop}\label{parabolic}
Let $T>0$ and suppose $\theta$ is the strong solution on $[0,T]$ corresponding to some $\theta_0\in W^{1,\infty}$. It follows that
\begin{equation}\label{holdnonlin}
\left|\nabla\theta(t,x)-\nabla\theta(t,z)\right|\lesssim \left(t^{-1/2(1+\beta)}+t^{1/2(1-\beta)}\right)|x-z|^\beta,
\end{equation}
where $\beta\in(0,1)$, $t\in(0,T]$ are arbitrary. Consequently, we get that $\partial_t\theta\in L^1\left([0,T];L^{\infty}(\mathbb{R}^d)\right)$, $\theta\in C_t^1C_x^2\left((0,T]\times\mathbb{R}^d\right)$, and 

\[
\partial_t\theta(t,x)-\nu\Delta\theta(t,x)-\lambda|\nabla\theta(t,x)|^p-\mu(-\Delta)^{\alpha}\theta(t,x)=0,
\]
holds true in the classical (pointwise) sense $\forall (t,x)\in(0,T]\times\mathbb{R}^d$.
\end{prop}
\begin{proof}
We have 
\begin{align*}
\nabla\theta(t,x)=&\int_{\mathbb{R}^d}\nabla\Psi(t,x-y)\theta_0(y)dy+\lambda\int_{0}^t\int_{\mathbb{R}^d}\nabla\Psi(t-s,x-y)\left|\nabla\theta(s,y)\right|^pdyds\\
&+\mu\int_{0}^t\int_{\mathbb{R}^d}\nabla\Psi(t-s,x-y)\left(-\Delta\right)^{\alpha}\theta(s,y)\ dy\ ds,
\end{align*}
and so \eqref{holdnonlin} follows from estimate \eqref{holdkern}, bound \eqref{nlabel2} with $\gamma=1$ from Lemma \ref{bdfraclap} and straightforward bounds.

To prove that $\partial_t\theta\in L^1\left([0,T];L^{\infty}(\mathbb{R}^d)\right)$, it is sufficient to show, for any $t_0\in[\epsilon,T/2]$,
\begin{equation}\label{dervttheta}
\|\partial_t\theta(2t_0,\cdot)\|_{L^{\infty}}\lesssim t_0^{-1/2}+1,
\end{equation}
with $\epsilon>0$ being arbitrarily small. First of all, notice that by virtue of Proposition \ref{wellposedness}, we must have, for any $t\in[t_0,T]$,
\[
\theta(t,x)=\varphi_0(t,x)+\varphi_1(t,x),
\]
where
\begin{align}
&\varphi_0(t,x):=\int_{\mathbb{R}^d}\Psi(t-t_0,x-y)\theta(t_0,y)\ dy,\label{defphi0}\\
&\varphi_1(t,x):=\mu\int_{t_0}^t\int_{\mathbb{R}^d}\Psi(t-s,x-y)\left(-\Delta\right)^{\alpha}\theta(s,y)\ dy\ ds\nonumber\\
&\qquad \qquad \quad +\lambda\int_{t_0}^t\int_{\mathbb{R}^d}\Psi(t-s,x-y)\left|\nabla\theta(s,y)\right|^pdy\ ds. \label{defphi1}
\end{align}
Differentiating $\varphi_0$ once in time, integrating by parts, bounding $|\nabla\theta(t_0,y)|\leq\|\nabla\theta(t_0,\cdot)\|_{L^{\infty}}$ and using \eqref{intgradkern} we see that, whenever $t\in(t_0,T]$,
\begin{equation}\label{dervphi0}
\|\partial_t\varphi_0(t,\cdot)\|_{L^{\infty}}\lesssim (t-t_0)^{-1/2}.
\end{equation}
As $\theta(t,\cdot)$ is Lipschitz, owing to Lemma \ref{fraclapmod}, we must have that 
\begin{equation}\label{frachold}
|(-\Delta)^{\alpha}\theta(s,x)-(-\Delta)^{\alpha}\theta(s,y)|\lesssim |x-y|^{1-2\alpha},
\end{equation}
whenever $s\in[0,T]$, $(x,y)\in\mathbb{R}^d$, while estimate \eqref{holdnonlin}, along with $|a^p-b^p|\leq p(a^{p-1}+b^{p-1})|a-b|$ tells us that
\begin{equation}\label{holdnonlin1}
\left||\nabla\theta(s,y)|^p-|\nabla\theta(s,x)|^p\right|\lesssim|x-y|^{\beta}\left(s^{-\frac{1}{2}(1+\beta)}+s^{1/2(1-\beta)}\right)\lesssim|x-y|^{\beta}\left(t_0^{-\frac{1}{2}(1+\beta)}+1\right),
\end{equation} 
whenever $s\in[t_0,T]$, $(x,y)\in\mathbb{R}^d$, and $\beta\in(0,1)$. H\"older estimates \eqref{frachold}-\eqref{holdnonlin1} allow us to differentiate the volume potentials \eqref{defphi1} once in time (see for instance \cite{Friedmanbook1964}) to get
\begin{align*}
\partial_t\varphi_1(t,x)=&\mu\left(-\Delta\right)^{\alpha}\theta(t,x)+\lambda\left|\nabla\theta(t,x)\right|^p\\
+&\mu\int_{t_0}^t\int_{\mathbb{R}^d}\partial_t\Psi(t-s,x-y)\left(\left(-\Delta\right)^{\alpha}\theta(s,y)-\left(-\Delta\right)^{\alpha}\theta(s,x)\right)dy\ ds\\
&+\lambda\int_{t_0}^t\int_{\mathbb{R}^d}\partial_t\Psi(t-s,x-y)\left(\left|\nabla\theta(s,y)\right|^p-\left|\nabla\theta(s,x)\right|^p\right)dy\ ds.
\end{align*}
Choosing $\beta=1-2\alpha$, bounding from above, and utilizing \eqref{intdtkern} we get
\begin{equation}\label{dervphi1}
\begin{aligned}
|\partial_t\varphi_1(t,x)|&\lesssim \left(t_0^{\alpha-1}+1\right)\int_{t_0}^t\int_{\mathbb{R}^d}|x-y|^{1-2\alpha}|\partial_t\Psi(t-s,x-y)|dy\ ds+1\\
&\lesssim\left(t_0^{\alpha-1}+1\right)(t-t_0)^{\frac{1}{2}(1-2\alpha)}+1.
\end{aligned}
\end{equation}
From \eqref{dervphi0} and \eqref{dervphi1}, we obtain \eqref{dervttheta}. It is clear that $\varphi_0$ is smooth and solves the homogenous heat equation on $[t_0,T]\times\mathbb{R}^d$, while the H\"older estimates \eqref{frachold}-\eqref{holdnonlin1} allow us to differentiate the volume potentials \eqref{defphi1} twice in space to conclude that $\varphi_1\in C_t^1C_x^2\left([t_0,T]\times\mathbb{R}^d\right)$ and
\[
\partial_t\varphi_1(t,x)-\nu\Delta\varphi_1(t,x)=\mu\left(-\Delta\right)^{\alpha}\theta(t,x)+\lambda|\nabla\theta(t,x)|^p,
\]
holds true in the pointwise 	sense on $[t_0,T]\times\mathbb{R}^d$, with $t_0\geq\epsilon>0$ being arbitrary small.
\end{proof}

\begin{prop}\label{criteria}
	Suppose $\theta$ is the strong solution on $[T_1,T_2]$ corresponding to some $\theta_0\in W^{1,\infty}$. If 
	\[
	T_*:=\sup\left\{T\geq T_1:\|\theta(t,\cdot)\|_{W^{1,\infty}\left(\mathbb{R}^d\right)}<\infty,\ \forall t\in[T_1,T]\right\},
	\]
	is the maximal interval of existence of the strong solution, then 
	\[
	T_*=\sup\left\{T\geq T_1:\|\nabla\theta(t,\cdot)\|_{L^{\infty}}<\infty,\ \forall t\in[T_1,T]\right\}.
	\] 
\end{prop}
\begin{proof}
Set 
\begin{align*}
&T_0:=\sup\left\{T\geq T_1:\|\theta(t,\cdot)\|_{L^{\infty}}<\infty, \ \forall t\in[T_1,T]\right\},\\
&\widetilde T_0:=\sup\left\{T\geq T_1:\|\nabla\theta(t,\cdot)\|_{L^{\infty}}<\infty, \ \forall t\in[T_1,T]\right\},
\end{align*}
and suppose for the sake of contradiction that $T_0<\widetilde T_0$. We must have 
\[
\limsup_{t\rightarrow T_0^{-}}\|\theta(t,\cdot)\|_{L^{\infty}}=\infty,
\]
while for $t\in[T_1,T_0)$,
\begin{align*}
\theta(t,x)=&\int_{\mathbb{R}^d}\Psi(t-T_1,y)\theta_0(x-y)\ dy+\lambda\int_{T_1}^t\int_{\mathbb{R}^d}\Psi(t-s,x-y)\left|\nabla\theta(s,y)\right|^pdyds\\
&+\mu\int_{T_1}^t\int_{\mathbb{R}^d}\Psi(t-s,x-y)\left(-\Delta\right)^{\alpha}\theta(s,y)\ dy\ ds.
\end{align*}
Setting 
\[
A:=\max_{t\in [T_1,T_0]}\left\|\nabla\theta(t,\cdot)\right\|_{L^{\infty}}<\infty,
\] 
and bounding from above, while utilizing Lemma \ref{bdfraclap}, we get, whenever $t\in[T_1,T_0)$,
\[
\left\|\theta(t,\cdot)\right\|_{L^{\infty}}\leq \|\theta_0\|_{L^{\infty}}+|\lambda|A^p(t-T_1)+\mu C_{d,\alpha}A^{2\alpha}\int_{T_1}^t\left\|\theta(s,\cdot)\right\|_{L^{\infty}}^{1-2\alpha}\ ds.
\]
As $1-2\alpha\in(0,1)$, we must have
\[
\left\|\theta(s,\cdot)\right\|_{L^{\infty}}^{1-2\alpha}\leq \left(1-2\alpha\right)\left\|\theta(s,\cdot)\right\|_{L^{\infty}}+2\alpha,
\]
allowing us to conclude the proof by Gronwall's inequality.	
\end{proof}

We conclude this section with a few comments. For starters, to prove an analogous result when $\alpha\in[1/2,1)$, we can use the same technique. One way of taking care of things is by requiring the initial data to be in $W^{2,\infty}(\rd)$ and find a fixed point in the space $X_T:=C\left([0,T];W^{2,\infty}(\mathbb{R}^d)\right)$ while utilizing bound \eqref{nlabel3} instead of \eqref{nlabel2} in Lemma \ref{bdfraclap}. The only part of the proof that has to be significantly changed is the regularity criterion (Proposition \ref{criteria}), and we could instead specialize to the periodic or whole space setting separately and work with energy estimates rather than pointwise.

The requirement that $\theta_0\in W^{1,\infty}(\rd)$ is not optimal: from Lemma \ref{bdfraclap} and the above proof, it shouldn't be too hard to work with $\theta_0\in L^{\infty}\cap C^{0,\gamma}$ with $\gamma\in(2\alpha,1]$. Of course we have to appropriately modify the definition of ``strong solutions'' along with the proof of Proposition \ref{construct} in order to make sense of the nonlinearity. In fact, using heat kernel properties, one can show that $\|\nabla\theta(t,\cdot)\|_{L^{\infty}}\lesssim t^{(\gamma-1)/2}[\theta_0]_{C^{0,\gamma}}$ for $t$ close to 0. We only need to make sure that the initial data has sufficiently high regularity to treat the nonlinear equation as a perturbation of the heat equation. If a nonlinear evolution equation is invariant under some scaling, a general rule of thumb is that one should expect a good local well-posedness theory in spaces with norms that are invariant (critical) or are monotone (subcritical) with respect to the scaling. Since equation \eqref{maineq} with $\mu=0$ does have a scale invariance, it is natural to expect a good local well-posedness theory in spaces that respect such invariance. Indeed, the nonlocal term is linear and of order less than dissipation, so it is not expected to dramatically change the local theory. We do not pursue that direction here.
\section{Proof of Theorem \ref{mainthm}}\label{proofmainthm}
\subsection{Strategy of the Proof.} 
By Theorem \ref{shortimeexis} and Proposition \ref{criteria}, we only need to show that $\|\nabla\theta(t,\cdot)\|_{L^{\infty}}<\infty$ for any $t\geq0$. This will be achieved by constructing a time-dependent strong modulus of continuity (an $\Omega(t,\xi)$ such that $\Omega(t,\cdot)$ is a strong modulus of continuity for any $t\geq0$ according to Definition \ref{defmod}) such that $\theta(t,\cdot)$ has $\Omega(t,\cdot)$ as a strict modulus of continuity for all $t\geq0$, thereby ruling out the gradient blowup scenario. As a byproduct, we are able to obtain an explicit bound on the gradient in terms of $\partial_\xi\Omega(t,0)$. 

 Time-dependent moduli of continuity have been studied before in \cite{Kiselev2011}, mainly in the context of eventual regularization of active scalars. H\"older time-dependent moduli of continuity were also considered in \cite{SV2012}, where a drift-diffusion equation with a pressure term was considered, and the solution is shown to remain H\"older continuous as long as the drift velocity is under control. Following (and slightly generalizing) the work in \cite{Kiselev2011}, the time-dependent modulus $\Omega(t,\xi)$ will be constructed such that the initial data has strict modulus of continuity $\Omega(0,\cdot)$ and 
\begin{equation}\label{heqmod}
\partial_t\Omega(t,\xi)-4\nu\partial^2_\xi\Omega(t,\xi)>h(t,\xi),\quad (t,\xi)\in(0,\infty)\times \mathbb{R}^+,
\end{equation}
where $h$ will represent any ``local'' instabilities that may arise from the nonlocal and nonlinear part of the equation, that may depend linearly, nonlinearly or nonlocally on $\Omega$.

As will be shown below, the nonlinear term will not be of any concern and will in fact vanish (as is expected when trying to prove ``maximum principles''); we only need to worry about the nonlocal term, and as is easily observed for the linear equation (that is, equation \eqref{maineq} with $\lambda=0$), this will cause at most exponential growth in time, but not blowup. As opposed to \cite{Kiselev2011}, where $h(t,\xi)$ is a ``nonlocal'' Burgers type nonlinear term, in our case, since the nonlinearity will vanish, $h(t,\xi)$ is a linear term in $\Omega$, allowing us to solve the ``heat inequality'' \eqref{heqmod} by a simple separation of variables, i.e. we seek a modulus of continuity of the form
\[
\Omega(t,\xi)=f(t)\omega(\xi).
\]
In our case,
\[
h(t,\xi)=C_{d,\alpha}\int_0^\xi\frac{\partial_\eta\Omega(t,\eta)}{\eta^{2\alpha}}\ d\eta,
\]
and owing to the fact that
\[
\lim_{\xi\rightarrow0^+}\partial^2_{\xi}\Omega(t,\xi)=-\infty,
\]
we see that the local dissipation from the Laplacian (the term $-4\nu\partial^2_\xi\Omega$) will balance out $h$ when $\xi$ is small. Time dependence on the other hand is necessary, since the above integral cannot be made to vanish as $\xi\rightarrow\infty$, while local dissipation from the Laplacian must go to 0 as $\xi\rightarrow\infty$, otherwise the modulus of continuity will become negative for large $\xi$. Therefore, we need to rely on the time derivative $\partial_t\Omega$ to balance out those instabilities when $\xi$ is large. This will be clear in \S\ref{constsec} below.

Before constructing the modulus of continuity, let us recall the main ideas introduced in \cite{KNS2008,KNV2007} and slightly modify them in order to be applicable for problem \eqref{maineq}. Let us suppose that $\Omega(t,\cdot)$ is an unbounded strong modulus of continuity for each $t\geq0$, and assume that $\theta_0$ has $\Omega(0,\xi)$ as a strict modulus of continuity. Furthermore, suppose that $\Omega\in C\left([0,\infty)\times[0,\infty)\right)$, and that $\Omega(\cdot,\xi)$ is non-decreasing as a function of time for each $\xi\in[0,\infty)$. Let us now define
\begin{align}
&T_*:=\sup\left\{T\geq 0:\|\theta(t,\cdot)\|_{W^{1,\infty}}<\infty, \ \forall t\in[0,T]\right\},\label{defmaxT}\\
&\tau:=\sup\left\{T\geq0:|\theta(t,x)-\theta(t,y)|<\Omega\left(t,|x-y|\right),\ \forall t\in[0,T],\ x\neq y\right\}\label{deftau},
\end{align}
where $x,y\in\mathbb{R}^d\times\mathbb{R}^d$ are arbitrary, and assume for the moment that $\tau>0$. It is clear that if $T_*<\infty$, then $\tau<T_*$: by virtue of Proposition \ref{criteria}, $\theta$ must exhibit gradient blowup at $T_*$, while the fact that $\partial_\xi\Omega(t,0)<\infty$ would lead to a uniform bound for the gradient on the interval $[0,T_*]$. It follows that by continuity, $\theta(\tau,\cdot)$ has $\Omega(\tau,\cdot)$ as a modulus of continuity, albeit not necessarily strict. The idea is then to construct $\Omega$ such that if $\tau>0$, then $\tau=\infty$, contradicting the fact that $T_*<\infty$ and providing the explicit bound 
\[
\left\|\nabla\theta(t,\cdot)\right\|_{L^{\infty}}<\partial_\xi\Omega(t,0),\ \forall t\geq0.
\] 
To show that $\tau=\infty$, it will be sufficient to rule out the ``breakthrough'' scenario 
\begin{equation}\label{breakthrough}
|\theta(\tau,x)-\theta(\tau,y)|=\Omega(\tau,|x-y|),\quad \text{any }x\neq y.
\end{equation}
Indeed, suppose scenario \eqref{breakthrough} is not possible. As $\tau<T_*$, the solution is still $C^2$ in space at time $\tau$, and for a short time beyond that. It follows that Lemma \ref{gradbd} is applicable, and so $\|\nabla\theta(\tau,\cdot)\|_{L^{\infty}}<\partial_\xi\Omega(\tau,0)$. This guarantees that the strict modulus of continuity can never be violated in a neighborhood of the diagonal $x=y$, and so the same must be true for a short time beyond time $\tau$ and small $|x-y|$, say $|x-y|\in(0,\delta)$ some $\delta>0$. Since the solution is bounded in space in a neighborhood of time $\tau$, while $\Omega$ is unbounded in space and nondecreasing in time guarantees that the strict modulus of continuity is not violated for a short time beyond $\tau$ and large $|x-y|$, say $|x-y|\geq K$, some $K>\delta$. The only troublesome case is extending the time $\tau$ when $|x-y|\in[\delta,K]$ without assuming any bound on $x$ or $y$. This can be done under the assumption that the solution is either periodic or vanishes at spatial infinity (both properties which are inherited from the initial data, Remark \ref{imprmk}). Let us now make this rigorous. No $C^2$ or concavity assumptions on $\Omega(t,\cdot)$ are necessary for the proof of the next proposition. See also \cite[Lemma~2.3]{Kiselev2011}.
\begin{prop}\label{brkthruscen}
Suppose $\Omega\in C\left([0,\infty)\times[0,\infty)\right)$ is such that $\Omega(t,\cdot)$ is an unbounded strong modulus of continuity for each $t\geq0$, and that $\Omega(\cdot,\xi)$ is nondecreasing as a function of time for each $\xi\geq0$. Suppose $\theta_0$ has $\Omega(0,\cdot)$ as a strict modulus of continuity, and let $\theta$ be the (short-time) strong solution to \eqref{maineq} corresponding to $\theta_0$. Assume further that either $\theta_0$ is periodic with period $L>0$ or vanishes at infinity, and let $T_*$ and $\tau$ be as defined in \eqref{defmaxT} and \eqref{deftau}, respectively. It follows that $\tau>0$ and if $\tau<\infty$, then we must have $|\theta(\tau,x)-\theta(\tau,y)|=\Omega(\tau, |x-y|)$ for some $x\neq y$.
\end{prop}
\begin{proof}
	By virtue of Theorem \ref{shortimeexis} we may, without any loss in generality, assume $\theta_0$ is $C^2(\rd)$. In this case, Lemma \ref{gradbd} tells us that $\|\nabla\theta_0\|_{L^{\infty}}<\partial_\xi\Omega(0,0)$, and by continuity of the function $\left\|\nabla\theta(t,\cdot)\right\|_{L^{\infty}}$, this remains true for $t\in[0,\epsilon_0]$, some $\epsilon_0>0$. Set
	\begin{align*}
	&M_0:=\max_{t\in[0,\epsilon_0]}\left\|\nabla\theta(t,\cdot)\right\|_{L^{\infty}}<\partial_\xi\Omega(0,0),\\
	&M_1:=\max_{t\in[0,\epsilon_0]}\left\|\theta(t,\cdot)\right\|_{L^{\infty}},
	\end{align*}
	and for $\xi>0$, consider the function
	\[
	h(\xi):=M_0-\frac{\Omega(0,\xi)}{\xi}.
	\]
	Clearly, $h(\xi)<0$ for $\xi\in(0,\delta)$, some $\delta>0$. It follows that whenever $t\in[0,\epsilon_0]$ and $|x-y|\in(0,\delta)$, since $\Omega(\cdot,\xi)$ is nondecreasing as a function of time for each fixed $\xi\geq0$, we must have
	\[
	|\theta(t,x)-\theta(t,y)|\leq M_0|x-y|<\Omega(t,|x-y|).
	\]
	As $\Omega(0,\cdot)$ is unbounded and nondecreasing, there exists some $K>>\delta$ such that $\Omega(0,\xi)\geq 3M_1$ whenever $\xi\geq K$. It follows that whenever $t\in[0,\epsilon_0]$ and $|x-y|\in[K,\infty)$, 
	\[
	|\theta(t,x)-\theta(t,y)|\leq 2M_1<3M_1\leq\Omega(0,|x-y|)\leq\Omega(t,|x-y|).
	\]
	It remains to handle the case $|x-y|\in[\delta,K]$. If $\theta_0$ is $L-$ periodic, then so is $\theta(t,\cdot)$ (Remark \ref{imprmk}) and in this case, we first define
	\[
	\mathcal{A}:=\left\{(x,y)\in[0,L]^d\times\mathbb{R}^d:|x-y|\in[\delta,K]\right\},
	\]
	and note that since the set $[0,\epsilon_0]\times \mathcal{A}$ is compact, the function
	\[
	R(t,x,y):=|\theta(t,x)-\theta(t,y)|-\Omega(t,|x-y|),
	\]
	is uniformly continuous on it, and as $R(0,x,y)<0$, the same must be true on $[0,\epsilon]\times \mathcal{A}$, some $\epsilon\in(0,\epsilon_0]$. As $\theta(t,\cdot)$ is $L-$ periodic, this proves that $\tau\geq\epsilon>0$. 
	
	On the other hand, if $\theta_0$ vanishes at spatial infinity, one can chose a large enough $K_0$ and a small enough $\epsilon_1\in(0,\epsilon_0]$ such that $|\theta(t,z)|\leq \Omega(0,\delta)/4$ whenever $|z|\geq K_0$, $t\in[0,\epsilon_1]$, owing to the fact that $\partial_t\theta\in L^1([0,T_*];L^{\infty}(\mathbb{R}^d))$. We now decompose the set
	\[
	\mathcal{B}:=\left\{(x,y)\in\mathbb{R}^d\times\mathbb{R}^d: |x-y|\in[\delta,K]\right\}
	\]
	into $\mathcal{B}_1\cup \mathcal{B}_2$ where
	\[
	\mathcal{B}_1:=\left\{(x,y)\in \mathcal{B}: \min\{|x|,|y|\}>K_0\right\},
	\]
	and $\mathcal{B}_2$ is the complement of the set $\mathcal{B}_1$. By choice of $K_0$ and $\epsilon_1$, we have 
	\[
	|\theta(t,x)-\theta(t,y)|<\Omega(0,\delta)\leq\Omega(t,|x-y|),\ (t,x,y)\in[0,\epsilon_1]\times \mathcal{B}_1,
	\]	
	since $\Omega$ is nondecreasing in both variables. It is fairly straightforward to verify that $\mathcal{B}_2$ is compact, and so as in the periodic case, one can find a small enough $\epsilon\in(0,\epsilon_1]$ such that $\Omega$ is not violated on $[0,\epsilon]$. 
	
	The second part of the proposition follows by similar arguments. As discussed previously, the solution has not exhibited any blowup on $[0,\tau]$, and so $\theta(\tau,\cdot)$ has modulus of continuity $\Omega(\tau,\cdot)$, albeit not necessarily strict. Therefore, Lemma \ref{gradbd} can still be applied and we have the strict bound $\left\|\nabla\theta(\tau,\cdot)\right\|_{L^{\infty}}<\partial_\xi\Omega(\tau,0)$. Since $\theta(\tau,\cdot)\in W^{1,\infty}\left(\mathbb{R}^d\right)$, by Theorem \ref{shortimeexis} and Remark \ref{imprmk}, the solution is smooth for a short time beyond $\tau$, and is periodic or vanishes at infinity if $\theta_0$ is. Therefore, assuming that
	\[
	|\theta(\tau,x)-\theta(\tau,y)|<\Omega(\tau,|x-y|),\ \forall x\neq y,
	\]
	allows us to repeat the above argument and prolong the time $\tau$, by some $\epsilon>0$, contradicting the definition of $\tau$. Hence, if $\tau<\infty$, we must have
	\[
	|\theta(\tau,x)-\theta(\tau,y)|=\Omega(\tau,|x-y|),
	\]
	for some $x\neq y$.
\end{proof}
\subsection{Constructing the Modulus of Continuity.}\label{constsec}
We start by analyzing the breakthrough scenario described in Proposition \ref{brkthruscen}, i.e. we assume $\tau$ as defined in \eqref{deftau} is positive and finite, so that
\[
|\theta(\tau,x^0)-\theta(\tau,y^0)|=\Omega(\tau,|x^0-y^0|),
\]
for some $x^0\neq y^0$. Since everything is rotation and translation invariant, we may assume that the strict modulus of continuity is violated at some $x^0\neq y^0$, with $x^0-y^0=(\xi,0,\cdots,0)$, some $\xi>0$. Further, it is sufficient to assume
\[
\theta(\tau,x^0)-\theta(\tau,y^0)=\Omega(\tau,\xi),
\]
the case when $\theta(\tau,x^0)<\theta(\tau,y^0)$ is handled similarly. To rule out this scenario, we consider the function
\[
g(t):=\theta(t,x^0)-\theta(t,y^0)-\Omega(t,\xi),
\]
on the interval $[0,\tau+\epsilon]$, some small enough $\epsilon$, and we construct $\Omega$ such that $g'(\tau)<0$, for any $\xi>0$ and for which $\theta_0$ strictly obeys $\Omega(0,\cdot)$. To do so, we start by obtaining a bound on $g'(\tau)$ by using the fact that the PDE holds pointwise to get that
\begin{align}\label{eqbrkthrgh}
g'(\tau)=&\nu\left(\Delta\theta(\tau,x^0)-\Delta\theta(\tau,y^0)\right)-\partial_t\Omega(\tau,\xi)\nonumber\\
&+\lambda\left|\nabla\theta(\tau,x^0)\right|^p-\lambda\left|\nabla\theta(\tau,y^0)\right|^p+\mu\left(-\Delta\right)^{\alpha}\theta(\tau,x^0)-\mu\left(-\Delta\right)^{\alpha}\theta(\tau,y^0).
\end{align}
The first line in equation \eqref{eqbrkthrgh} is of stabilizing nature, while the second may cause instabilities. From \eqref{derv}, we see that 
\[
\left|\nabla\theta(\tau,x^0)\right|^p-\left|\nabla\theta(\tau,y^0)\right|^p=0,
\]
while \eqref{lap} and \eqref{modlap} give us
\begin{align*}
&\nu\left(\Delta\theta(\tau,x^0)-(\Delta\theta(\tau,y^0)\right)+\mu(-\Delta)^{\alpha}\theta(\tau,x^0)-\mu(-\Delta)^{\alpha}\theta(\tau,y^0)\\
&\leq 4\nu\partial^2_\xi\Omega(\tau,\xi)+\mu C_{d,\alpha}\int_0^\xi\frac{\partial_\eta\Omega(\tau,\eta)}{\eta^{2\alpha}}\ d\eta.
\end{align*}
Therefore, we obtain
\begin{equation*}
g'(\tau)\leq 4\nu\partial^2_\xi\Omega(\tau,\xi)-\partial_t\Omega(\tau,\xi)+\mu C_{d,\alpha}\int_0^\xi\frac{\partial_\eta\Omega(\tau,\eta)}{\eta^{2\alpha}}\ d\eta,
\end{equation*}
and so our aim is to construct an $\Omega$ that satisfies the hypothesis of Proposition \ref{brkthruscen} and for which 
\begin{equation}\label{gprime}
	\partial_t\Omega(t,\xi)-4\nu\partial^2_\xi\Omega(t,\xi)-\mu C_{d,\alpha}\int_0^\xi\frac{\partial_\eta\Omega(t,\eta)}{\eta^{2\alpha}}\ d\eta>0,
\end{equation}
for every $(t,\xi)\in (0,\infty)\times(0,\infty)$. To that extent, we start by defining
\[
\omega(\xi):=\frac{\xi}{1+\xi^{1-\alpha}},\quad \xi\geq0.
\]
Clearly, $\omega$ is an unbounded strong modulus of continuity: it is concave, grows like $\xi^{\alpha}$, $\omega'(0)=1$ and $\omega''(\xi)=-O(\xi^{-\alpha})$ near $\xi=0$. Next, chose a sufficiently large $B=B\left(\|\theta_0\|_{W^{1,\infty}}\right)$ such that $\theta_0$ has $\omega_B(\xi):=\omega(B\xi)$ as a strict modulus of continuity (owing to Lemma \ref{buildmod}), and let $\delta_0=\delta_0(B,\nu,\alpha,\mu,d)>0$ be a small number to be determined later. Set 
\[
C_0:=\frac{\mu C_{d,\alpha}}{\omega_B(\delta_0)}\left(\frac{B^{2\alpha-1}}{\delta_0}+\frac{B^{\alpha}}{\delta_0^{\alpha}}+\frac{B\delta_0^{1-2\alpha}}{(1-2\alpha)}\right),
\]
and let $f(t)=\exp(C_0t)$, which solves
\begin{equation}\label{deff}
f'(t)-C_0f(t)=0,\quad f(0)=1.	
\end{equation}
Finally, define
\[
\Omega(t,\xi):=f(t)\omega_B(\xi), \quad (t,\xi)\in[0,\infty)\times[0,\infty),
\]
and note that $\Omega$ satisfies the hypothesis of Proposition \ref{brkthruscen}.  Now, $\delta_0$ will be chosen small enough such that the dissipative term alone will balance the local instabilities arising from the nonlocal part for $\xi\in(0,\delta_0]$, while the time-dependent part of $\Omega$ will balance those instabilities away from $\delta_0$. To see this, as $\omega$ is concave, we have $\omega_B'(\xi)\leq B$, and so
\[
\int_0^\xi\frac{\partial_\eta\Omega(t,\eta)}{\eta^{2\alpha}}\ d\eta\leq f(t)(1-2\alpha)^{-1}B\xi^{1-2\alpha},
\] 
and as $\displaystyle{\lim_{\xi\rightarrow0^+}\omega_B''(\xi)=-\infty}$, one can choose a $\delta_0=\delta_0(B,\nu,\alpha,\mu,d)>0$ such that 
\[
4\nu\omega_B''(\xi)+\mu C_{d,\alpha}(1-2\alpha)^{-1}B\xi^{1-2\alpha}<0,
\]
whenever $\xi\in(0,\delta_0]$. 
A straightforward calculation yields that 
\[
\delta_0\leq \min\left\{B^{-1},\left(\frac{\nu(1-2\alpha)}{8\mu C_{d,\alpha}}\right)^{\frac{1}{(1-\alpha)}}\right\}.
\]
This immediately implies that \eqref{gprime} is true for any $(t,\xi)\in[0,\infty)\times(0,\delta_0]$, as $f$ is positive and nondecreasing. When $\xi\geq\delta_0$ our aim is to bound \eqref{modlap} uniformly in $\xi$, and so we use the bound 
\[
\omega_B'(\eta)\leq B^{2\alpha-1}\eta^{2\alpha-2}+\alpha B^{\alpha}\eta^{\alpha-1},
\]
to get 
\[
\int_{\delta_0}^\xi\frac{\partial_\eta\Omega(t,\eta)}{\eta^{2\alpha}}\ d\eta\leq f(t)\left(\frac{B^{2\alpha-1}}{\delta_0}+\frac{B^{\alpha}}{\delta_0^{\alpha}}\right).
\]
Therefore, whenever $(t,\xi)\in[0,\infty)\times(\delta_0,\infty)$, using the fact that $\omega_B$ is concave, we can bound the left-hand side of \eqref{gprime} from below by  
\[
f'(t)\omega_B(\delta_0)-f(t)\mu C_{d,\alpha}\left(\frac{B^{2\alpha-1}}{\delta_0}+\frac{B^{\alpha}}{\delta_0^{\alpha}}+\frac{B\delta_0^{1-2\alpha}}{(1-2\alpha)}\right),
\]
making \eqref{gprime} true for $(t,\xi)\in[0,\infty)\times(\delta_0,\infty)$ by choice of $f$ \eqref{deff}, thereby concluding the proof of Theorem \ref{mainthm}.

\section{Concluding Remarks}

Let us start by commenting on the exponential growth observed in bound \eqref{bdlipsch}. As opposed to the scenario in the SQG and critical Burgers equation analyzed in \cite{KNS2008,KNV2007}, the instabilities in our case manifest themselves in estimate \eqref{modlap}, which cannot be made to decay in $\xi$. This is the main technical difficulty that forces us to allow the modulus to depend on time, as the best we could do is construct a modulus such that \eqref{modlap} is bounded. Therefore, if we do not ``absorb'' that term by a function of time, one would require the concavity of the modulus to be bounded from above by some fixed negative constant, since this will be the only positive quantity in inequality \eqref{gprime}. But this immediately implies that at some large enough $\xi$, the modulus becomes decreasing, and in fact, negative at even larger $\xi$. One might be able to overcome this in the periodic setting by constructing a more sophisticated modulus, since one only needs to rule out the ``breakthrough'' scenario for $\xi$ in some compact set in this case. 

A possible approach to prove that regularity persists under evolution when $\alpha=1/2$, and to eliminate time dependence in \eqref{bdlipsch}, is the following. Recall that the main difficulty in studying evolution of moduli of continuity under the original MS model \eqref{MS} is the lack of pointwise control of $(-\Delta)^{1/2}\theta$. However, one can bootstrap control of the Lipschitz constant, and obtain, via energy techniques, a bound on a high enough Sobolev norm. Owing to the Sobolev embedding theorem, we obtain a pointwise bound or even continuity (H\"older) estimate for the term $(-\Delta)^{1/2}\theta$ in \eqref{eqbrkthrgh}. Nevertheless, the time-dependent part of the modulus will now have to satisfy a first order ODE whose solution blows up in finite time, rendering the separation of variables approach useless.

On the other hand, Lemma \ref{gradbd} is still valid for moduli of continuity of the form
\[
\Omega(t,\xi):=\begin{cases}
\omega_{\mathcal{L}}(\xi),&\xi\in[0,\delta],\\
\omega_{\mathcal{R}}(t,\xi), &\xi\in(\delta,\infty),
\end{cases}
\] 
allowing for uniform-in-time control over the Lipschitz constant. However, solving the heat equation by a simple separation of variables for the right part of the modulus, $\omega_{\mathcal{R}}$, now results in a jump discontinuity at $\delta$, introducing various technical difficulties in the proof. Upon sharing the above remarks and results with the author's current doctoral adviser Edriss Titi, it was suggested that in order to adapt the approach for the MS model \eqref{MS}, one should instead try to solve the heat equation implicitly to patch the break in the modulus at $\delta$ \cite{TitiPrivateComm}. That is, solve a boundary value problem for $\omega_{\mathcal{R}}$. Ideally, one would want the modulus to be at least $C^1$ in space, and so this amounts to prescribing Cauchy data to the forced heat equation at $\xi=\delta$, resulting in an overdetermined problem. It was also suggested by Titi to relax one of the boundary conditions instead of trying to solve the overdetermined problem. The main technical difficulty in solving this so-called lateral Cauchy problem is the lack of a minimum principle, in particular one can no longer guarantee positivity of the modulus. That being said, it is natural to relax the Neumann condition, in order to guarantee concavity of the modulus, at least near 0. This is remaining faithful to the spirit of the ideas presented in \cite{Kiselev2011, KNS2008,KN2010,KNV2007}, mainly in order to be able to deal with the nonlinearity and extract dissipation at $\delta$ (apply Lemma \ref{dervh}). In simple words, this translates to showing that the Dirichlet to Neumann map is not increasing in time, at least on an arbitrary interval of time $[0,T]$. This is currently under investigation by the author, where in the thesis \cite{Ibdah2020}, we study this in more details, as well apply this technique to other nonlinear-nonlocal PDEs. 

In light of trying to address the case when $\alpha=1/2$, we briefly discuss the recent work of Miao and Xue \cite{MX2019} (building upon \cite{KN2010, Kiselev2011}) which was brought to our attention by one of the anonymous referees. In the aforementioned paper \cite{MX2019}, the authors analyzed the following one-dimensional dissipative-dispersive perturbation of Burgers equation 
\begin{equation}\label{MX}
\partial_tu+(-\Delta)^{\gamma/2}u=u\partial_xu+L_\beta u,\quad \gamma\in[\beta,2],
\end{equation}
with $L_\beta$ being an operator of order at most $\beta$ with an odd kernel. For the main results of their work, they assumed $\beta\in(0,1)$; however a key estimate that was derived (and utilized for the case when $\gamma=1$) in \cite{MX2019} is the following bound: for any $\gamma\in[\beta,2)$ (including $\beta=1$), if $\xi:=|x-y|$ and $u(t,x)-u(t,y)=\Omega(t,\xi)$, then for some positive $C>0$ depending on $\gamma$ and $\beta$, the authors show that
\begin{equation}\label{bdMX}
L_\beta u(t,x)-L_\beta u(t,y)\leq -C\xi^{\gamma-\beta}D_\gamma+C\xi\int_\xi^{\infty}\frac{\Omega(t,\eta)}{\eta^{2+\beta}}d\eta+C\frac{\Omega(t,\xi)}{\xi^{\beta}},
\end{equation}
where $D_\gamma=(-\Delta)^{\gamma/2}u(t,y)-(-\Delta)^{\gamma/2}u(t,x)$. Notice that as in Lemma \ref{dervh} in this paper, $D_\gamma$ is strictly negative. The estimate \eqref{bdMX} is indeed surprising when $\beta=1$: one should not expect a continuity estimate on $Lu$ from $u$ when $L$ is an operator of order 1 or higher. Estimate \eqref{bdMX} does not violate this general rule, since it is valid only when we are at the breakthrough scenario depicted by Proposition \ref{brkthruscen}, and not for any $x,y$. Since we only care about the breakthrough scenario in our analysis, it is natural to ask if we can make use of such an estimate, by replacing classical dissipation with fractional, and adopting the approach of \cite{MX2019} here.

Unfortunately, we do not believe this can be done here, regardless whether dissipation is fractional or classical. Indeed, the strategy of \cite{MX2019} is to upgrade the regularity in steps: from $L_t^{\infty}L_x^2$ to $L_t^{\infty}L_x^\infty$ to $L_t^\infty C_x^{0,\delta}$, some $\delta\in(0,1]$, with the propagation of moduli of continuity (and utilizing bound \eqref{bdMX}) being applied in going from $L_t^{\infty}L_x^\infty$ to $L_t^\infty C_x^{0,\delta}$. If we ignore the term $L_\beta$, we end up with the critically dissipative fractional Burgers equation when $\gamma=1$, and so the bound $L_t^{\infty}L_x^{\infty}$ is not sufficient to deduce regularity, which is why the authors of \cite{MX2019} needed to go from $L_t^{\infty}L_x^\infty$ to $L_t^\infty C_x^{0,\delta}$. As a consequence of the $L_t^{\infty}L_x^{\infty}$ bound, one only needs to rule out the equality $u(t,x)-u(t,y)=\Omega(t,\xi)$ for $\xi$ in some bounded set, and hence by carefully constructing a stationary (independent of time) $\Omega$ that depends on the $L_t^{\infty}L_x^{\infty}$ bound, the authors were able to absorb the term $\xi^{\gamma-\beta}D_\gamma$ in the viscous one.

We do not have such luxury here, in particular, there is no a-priori $L_t^{\infty}L_x^{2}$ bound to bootstrap to $L_t^{\infty}L_x^{\infty}$ as in \cite[Lemma~2.6]{MX2019} (even if we consider the evolution of $u=\nabla\theta$ and $p=2$ to end up with the transport version of Burgers equation in multi-dimensions). In fact, one of the reasons that we were interested in the MS model is the lack of any obvious a-priori bounds. Unless the approach of \cite{MX2019} can be modified to bypass the $L_t^{\infty}L_x^{2}$ estimate, it does not seem to be applicable here. Actually if the proof of \cite[Lemma~2.6]{MX2019} can be modified to bypass the energy estimate, then we can probably apply it directly to the evolution of $u=\nabla\theta$ to obtain the required estimate in our scenario without having to propagate moduli of continuity. In fact, using similar arguments utilized in deducing regularity to the subcritical Burgers and SQG equation from the $L_t^{\infty}L_x^{\infty}$ bound, it will also be possible to deduce global regularity for the model
\[
\partial_tu+(-\Delta u)^{\gamma/2}=(u\cdot\nabla) u+(-\Delta)^{1/2}u,\quad \gamma\in(1,2),
\]
even when $u$ is not conservative ($u\neq\nabla\theta$), if one can improve on \cite[Lemma~2.6]{MX2019} as described above. This is because the moment we get an $L^{\infty}$ bound, we will be able to close the $H^s$ energy estimates (provided $\gamma>1$) by standard product and interpolation inequalities, and the linear, nonlocal one would not introduce dramatic difficulties (again, as long as $\gamma>1$).

So the next question is whether bound \eqref{bdMX} can be utilized directly in studying the evolution of $\Omega$, keeping in mind that we need to rule out the breakthrough scenario for $\xi\in(0,\infty)$ and not just in some bounded set, due to the lack of an $L_t^\infty L_x^{\infty}$ bound. The viscous term (fractional or classical) is very powerful over small distances, so the difficulty is when $\xi\in(1,\infty)$. The term $-C\xi^{\gamma-\beta}D_\gamma$ cannot be absorbed by the viscous one for $\xi$ in this region (even with fractional dissipation, $D_\gamma$). It also cannot be absorbed by $\partial_t\Omega$, since we need an \emph{upper bound} on $-D_\gamma$ in terms of $\Omega$ with $\gamma>1$, which is not available (what we have is a lower bound on $-D_\gamma$, see \cite{Kiselev2011}). That is to say, we are back to square zero: lack of a continuity estimate on an operator of order one or higher. That being said, a possible scenario where one can use bound \eqref{bdMX} to address the case $\alpha=1/2$ in our work is if we replace standard dissipation with fractional, restrict ourselves to the periodic setting and impose a smallness condition on the period. In this scenario, although we still do not have an $L_t^{\infty}L_x^{\infty}$ bound, periodicity tells us that we only need to worry about $\xi\in(0,\kappa]$, where $\kappa$ would depend on the period. Then, we can use bound \eqref{bdMX} with $\beta=1$ and $\gamma\in (1,2)$ to see that if the period is small enough, we will be able to absorb the term $-C\xi^{\gamma-\beta}D_\gamma$ in the dissipative one, provided the latter is coming from $(-\Delta)^{\gamma/2}$. The leftover terms can be handled by the time derivative if necessary. We also want to point out that the fractional Laplacian does not have an odd kernel, an assumption that was made on the operator $L_\beta$ in \cite{MX2019}. But this probably is not the main issue at hand, one only has to verify that all the key estimates do not rely on this cancellation property (which they probably do not). The main issue is the lack of energy bounds.

While we are on the topic of fractional dissipation, our final remark is that one would expect results analogous to those obtained in this work to hold even if the dissipative operator, $(-\Delta)$ is replaced by a fractional one $(-\Delta)^{\gamma}$, where $\gamma\in[1/2,1]$ (we switched the power to $\gamma$ instead of the $\gamma/2$ used in the previous paragraph to be consistent with the notation in \cite{Kiselev2011}). Indeed, it was shown in \cite{Kiselev2011}, that the local dissipative power of $(-\Delta)^{\gamma}$ for small $\xi$ is, roughly speaking, 
\[
\xi^{2-2\gamma}\omega''(\xi),
\]
and as long as $\gamma\geq1/2$, one can still construct a modulus of continuity, according to Definition \ref{defmod}, such that 
\[
\lim_{\xi\rightarrow0^+}\xi^{1-2\gamma+2\alpha}\omega''(\xi)=-\infty.
\]
To obtain a local well-posedness result and regularity criteria in terms of $\|\nabla\theta\|_{L^{\infty}}$ in this case, one should be able to adapt the ideas in \cite{DI2006, I2005,Silvestre2011} and the references therein.
\section*{Acknowledgements}
The author would like to thank the anonymous referees for their useful comments and feedback. This work is part of my doctoral thesis, and so I would like to thank my adviser Edriss Titi for introducing me to the Michelson-Sivashinsky equation, which lead to this work, as well as an upcoming current work in progress. I also thank Titi for several discussions. My gratitude is extended to Peter Kuchment for several editorial remarks.
\bibliographystyle{abbrv}
\bibliography{mybib}

\begin{thebibliography}{10}

\bibitem{Ben-ArtziAmour1998}
L.~Amour and M.~Ben-Artzi.
\newblock Global existence and decay for viscous {H}amilton-{J}acobi equations.
\newblock {\em Nonlinear Anal.}, 31(5-6):621--628, 1998.

\bibitem{BBT2003}
H.~Bellout, S.~Benachour, and E.~S. Titi.
\newblock Finite-time singularity versus global regularity for hyper-viscous
  {H}amilton-{J}acobi-like equations.
\newblock {\em Nonlinearity}, 16(6):1967--1989, 2003.

\bibitem{BAMJL2000}
M.~Ben-Artzi, J.~Goodman, and A.~Levy.
\newblock Remarks on a nonlinear parabolic equation.
\newblock {\em Trans. Amer. Math. Soc.}, 352(2):731--751, 2000.

\bibitem{CaffarelliSilvestre2007}
L.~Caffarelli and L.~Silvestre.
\newblock An extension problem related to the fractional {L}aplacian.
\newblock {\em Comm. Partial Differential Equations}, 32(7-9):1245--1260, 2007.

\bibitem{CC2004}
A.~C\'{o}rdoba and D.~C\'{o}rdoba.
\newblock A maximum principle applied to quasi-geostrophic equations.
\newblock {\em Comm. Math. Phys.}, 249(3):511--528, 2004.

\bibitem{DKSV2014}
M.~Dabkowski, A.~Kiselev, L.~Silvestre, and V.~Vicol.
\newblock Global well-posedness of slightly supercritical active scalar
  equations.
\newblock {\em Anal. PDE}, 7(1):43--72, 2014.

\bibitem{DI2006}
J.~Droniou and C.~Imbert.
\newblock Fractal first-order partial differential equations.
\newblock {\em Arch. Ration. Mech. Anal.}, 182(2):299--331, 2006.

\bibitem{EvansPDE2010}
L.~C. Evans.
\newblock {\em Partial differential equations}, volume~19 of {\em Graduate
  Studies in Mathematics}.
\newblock American Mathematical Society, Providence, RI, second edition, 2010.

\bibitem{Friedmanbook1964}
A.~Friedman.
\newblock {\em Partial differential equations of parabolic type}.
\newblock Prentice-Hall, Inc., Englewood Cliffs, N.J., 1964.

\bibitem{GS1990}
S.~Gutman and G.~Sivashinsky.
\newblock The cellular nature of hydrodynamic flame instability.
\newblock {\em Phys. D}, 43(1):129 -- 139, 1990.

\bibitem{Ibdah2020}
H.~Ibdah.
\newblock On preservation of moduli of continuity by parabolic evolution.
\newblock In preparation, 2021.

\bibitem{I2005}
C.~Imbert.
\newblock A non-local regularization of first order {H}amilton-{J}acobi
  equations.
\newblock {\em J. Differential Equations}, 211(1):218--246, 2005.

\bibitem{Kiselev2011}
A.~Kiselev.
\newblock Nonlocal maximum principles for active scalars.
\newblock {\em Adv. Math.}, 227(5):1806--1826, 2011.

\bibitem{KN2010}
A.~Kiselev and F.~Nazarov.
\newblock Global regularity for the critical dispersive dissipative surface
  quasi-geostrophic equation.
\newblock {\em Nonlinearity}, 23(3):549--554, 2010.

\bibitem{KNS2008}
A.~Kiselev, F.~Nazarov, and R.~Shterenberg.
\newblock Blow up and regularity for fractal {B}urgers equation.
\newblock {\em Dyn. Partial Differ. Equ.}, 5(3):211--240, 2008.

\bibitem{KNV2007}
A.~Kiselev, F.~Nazarov, and A.~Volberg.
\newblock Global well-posedness for the critical 2{D} dissipative
  quasi-geostrophic equation.
\newblock {\em Invent. Math.}, 167(3):445--453, 2007.

\bibitem{KO2013}
O.~Kupervasser and Z.~Olami.
\newblock Random noise and pole-dynamics in unstable front propagation.
\newblock {\em Combustion, Explosion, and Shock Waves}, 49(2):141--152, Mar
  2013.

\bibitem{Kuramoto1978}
Y.~Kuramoto.
\newblock Diffusion-induced chaos in reaction systems.
\newblock {\em Progress of Theoretical Physics Supplement}, 64:346--367, 02
  1978.

\bibitem{KuramotoTsuzuki1975}
Y.~Kuramoto and T.~Tsuzuki.
\newblock On the formation of dissipative structures in reaction-diffusion
  systems: Reductive perturbation approach.
\newblock {\em Progress of Theoretical Physics}, 54(3):687--699, 09 1975.

\bibitem{KuramotoTsuzuki1976}
Y.~Kuramoto and T.~Tsuzuki.
\newblock Persistent propagation of concentration waves in dissipative media
  far from thermal equilibrium.
\newblock {\em Progress of Theoretical Physics}, 55(2):356--369, 02 1976.

\bibitem{LariosYamazaki2019}
A.~Larios and K.~Yamazaki.
\newblock On the well-posedness of an anisotropically-reduced two-dimensional
  {K}uramoto-{S}ivashinsky equation.
\newblock {\em Phys. D}, 411:132560, 14, 2020.

\bibitem{Matalon2007}
M.~Matalon.
\newblock Intrinsic flame instabilities in premixed and nonpremixed combustion.
\newblock {\em Annual Review of Fluid Mechanics}, 39(1):163--191, 2007.

\bibitem{MX2019}
Q.~Miao and L.~Xue.
\newblock Regularity and singularity results for the dissipative {W}hitham
  equation and related surface wave equations.
\newblock {\em Commun. Math. Sci.}, 17(8):2141--2190, 2019.

\bibitem{MS1982}
D.~Michelson and G.~Sivashinsky.
\newblock Thermal-expansion induced cellular flames.
\newblock {\em Combustion and Flame}, 48:211 -- 217, 1982.

\bibitem{MichelsonSivashinsky1977}
D.~M. Michelson and G.~I. Sivashinsky.
\newblock Nonlinear analysis of hydrodynamic instability in laminar flames.
  {II}. {N}umerical experiments.
\newblock {\em Acta Astronaut.}, 4(11-12):1207--1221, 1977.

\bibitem{Navin2010}
F.~Navin.
\newblock Interplay between background turbulence and darrieus-landau
  instability in premixed flames via a model equation.
\newblock Master's thesis, University of Illinois at Urbana-Champaign, May
  2010.

\bibitem{NicoSheurer1984}
B.~Nicolaenko and B.~Scheurer.
\newblock Remarks on the {K}uramoto-{S}ivashinsky equation.
\newblock {\em Phys. D}, 12(1-3):391--395, 1984.

\bibitem{OGKP1997}
Z.~Olami, B.~Galanti, O.~Kupervasser, and I.~Procaccia.
\newblock Random noise and pole dynamics in unstable front propagation.
\newblock {\em Phys. Rev. E}, 55:2649--2663, Mar 1997.

\bibitem{Pumir1985}
A.~Pumir.
\newblock Equation describing wrinkled flame fronts.
\newblock {\em Phys. Rev. A}, 31:543--546, Jan 1985.

\bibitem{PSbook2019}
P.~Quittner and P.~Souplet.
\newblock {\em Superlinear parabolic problems}.
\newblock Birkh\"{a}user Advanced Texts: Basler Lehrb\"{u}cher. [Birkh\"{a}user
  Advanced Texts: Basel Textbooks]. Birkh\"{a}user/Springer, Cham, 2019.
\newblock Blow-up, global existence and steady states, Second edition of [
  MR2346798].

\bibitem{Renardy1987}
M.~Renardy.
\newblock A model equation in combustion theory exhibiting an infinite number
  of secondary bifurcations.
\newblock {\em Phys. D}, 28(1-2):155--167, 1987.

\bibitem{Silvestre2007}
L.~Silvestre.
\newblock Regularity of the obstacle problem for a fractional power of the
  {L}aplace operator.
\newblock {\em Comm. Pure Appl. Math.}, 60(1):67--112, 2007.

\bibitem{Silvestre2011}
L.~Silvestre.
\newblock On the differentiability of the solution to the {H}amilton-{J}acobi
  equation with critical fractional diffusion.
\newblock {\em Adv. Math.}, 226(2):2020--2039, 2011.

\bibitem{SV2012}
L.~Silvestre and V.~Vicol.
\newblock H\"{o}lder continuity for a drift-diffusion equation with pressure.
\newblock {\em Ann. Inst. H. Poincar\'{e} Anal. Non Lin\'{e}aire},
  29(4):637--652, 2012.

\bibitem{Sivashinsky1977}
G.~I. Sivashinsky.
\newblock Nonlinear analysis of hydrodynamic instability in laminar flames.
  {I}. {D}erivation of basic equations.
\newblock {\em Acta Astronaut.}, 4(11-12):1177--1206, 1977.

\bibitem{Sivashinsky1983}
G.~I. Sivashinsky.
\newblock Instabilities, pattern formation, and turbulence in flames.
\newblock {\em Annual Review of Fluid Mechanics}, 15(1):179--199, 1983.

\bibitem{Stein1961}
E.~M. Stein.
\newblock The characterization of functions arising as potentials.
\newblock {\em Bull. Amer. Math. Soc.}, 67:102--104, 1961.

\bibitem{Stein1970book}
E.~M. Stein.
\newblock {\em Singular integrals and differentiability properties of
  functions}.
\newblock Princeton Mathematical Series, No. 30. Princeton University Press,
  Princeton, N.J., 1970.

\bibitem{Tadmor1986}
E.~Tadmor.
\newblock The well-posedness of the {K}uramoto-{S}ivashinsky equation.
\newblock {\em SIAM J. Math. Anal.}, 17(4):884--893, 1986.

\bibitem{TFH1988489}
O.~Thual, U.~Frisch, and M.~Hénon.
\newblock Application of pole decomposition to an equation governing the
  dynamics of wrinkled flame fronts.
\newblock In P.~Pelcé, editor, {\em Dynamics of Curved Fronts}, pages 489 --
  498. Academic Press, San Diego, 1988.

\bibitem{TitiPrivateComm}
E.~S. Titi.
\newblock {Private Communication}, 2019.

\bibitem{TP2019}
A.~Trucchia and G.~Pagnini.
\newblock Restoring property of the {M}ichelson–{S}ivashinsky equation.
\newblock {\em Combustion Science and Technology}, 191(9):1734--1741, 2019.

\bibitem{Wheeden1968}
R.~L. Wheeden.
\newblock On hypersingular integrals and {L}ebesgue spaces of differentiable
  functions.
\newblock {\em Trans. Amer. Math. Soc.}, 134:421--435, 1968.

\end{thebibliography}

\end{document}